\documentclass[a4paper,11pt,reqno]{amsart}
\usepackage{amsmath}
\usepackage{amsthm,enumerate}
\usepackage{mathtools}
\usepackage{enumitem}

\usepackage{graphicx}
\usepackage{amssymb}

\usepackage[toc,page]{appendix}

\usepackage[utf8]{inputenc}

\usepackage{color}

\usepackage{url}

\usepackage[text={6.5in,8.6in},centering]{geometry}

\usepackage{srcltx}

\usepackage[T1]{fontenc}

\theoremstyle{plain}
\newtheorem{thm}{Theorem}[section]
\theoremstyle{plain}
\newtheorem{lem}[thm]{Lemma}
\newtheorem{prop}[thm]{Proposition}
\newtheorem{cor}[thm]{Corollary}

\theoremstyle{definition}
\newtheorem{defi}{Definition}[section]
\newtheorem{rem}{Remark}[section]

\newcommand{\hn}{\mathbb{H}^{N}}

\newcommand{\RR}{\mathbb{R}}
\newcommand{\HH}{\mathbb{H}}
\newcommand{\GM}{\mathbb{G}_{\HH^N}^s}

\newcommand{\GH}{\mathbb{G}_{\HH^N}^s}
\newcommand{\ka}{\overline{\kappa}}

\newcommand{\rd}{{\rm d}}

\newcommand{\dg}{\rd\mu_{\HH^N}}

\newcommand{\dgh}{\rd\mu_{\HH^N}}

\numberwithin{equation}{section} \allowdisplaybreaks

\begin{document}
 \title[Fractional Porous Medium Equation on the hyperbolic space]{The Fractional Porous Medium Equation\\ on the hyperbolic space}
\author{Elvise Berchio}
\address{\hbox{\parbox{5.7in}{\medskip\noindent{Dipartimento di Scienze Matematiche, \\
Politecnico di Torino,\\
        Corso Duca degli Abruzzi 24, 10129 Torino, Italy. \\[3pt]
        \em{E-mail address: }{\tt elvise.berchio@polito.it}}}}}

\author{Matteo Bonforte}
\address{\hbox{\parbox{5.7in}{\medskip\noindent{Departamento de Matem\'aticas,   Universidad Aut\'onoma de Madrid, and \\
ICMAT - Instituto de Ciencias Matem\'{a}ticas, CSIC-UAM-UC3M-UCM\\ 
Campus de Cantoblanco, 28049 Madrid, Spain
 \\[3pt]
        \em{E-mail address: }{\tt matteo.bonforte@uam.es}}}}}

        \author[Debdip Ganguly]{Debdip Ganguly}
\address{\hbox{\parbox{5.7in}{\medskip\noindent{Department of Mathematics,\\
 Indian Institute of Technology Delhi,\\
 IIT Campus, Hauz Khas, New Delhi\\
        Delhi 110016, India. \\[3pt]
        \em{E-mail address: }{\tt debdipmath@gmail.com}}}}}

\author{Gabriele Grillo}
\address{\hbox{\parbox{5.7in}{\medskip\noindent{Dipartimento di Matematica,\\
Politecnico di Milano,\\
   Piazza Leonardo da Vinci 32, 20133 Milano, Italy. \\[3pt]
        \em{E-mail address: }{\tt
          gabriele.grillo@polimi.it}}}}}

\date{}





\keywords{Fractional porous medium equation; Hyperbolic space; a priori estimates; smoothing effects}

\subjclass[2010]{Primary: 35R01. Secondary: 35K65, 35A01, 35R11, 58J35.}

\begin{abstract}
We consider a nonlinear degenerate parabolic equation of porous medium type, whose diffusion is driven by the (spectral) fractional Laplacian on the hyperbolic space.  We provide existence results for solutions, in an appropriate weak sense, for data belonging either to the usual $L^p$ spaces or to larger (weighted) spaces determined either in terms of a ground state of $\Delta_{\hn}$, or of the (fractional) Green's function.  For such solutions, we also prove different kind of smoothing effects, in the form of quantitative $L^1-L^\infty$ estimates. To the best of our knowledge, this seems the first time in which the fractional porous medium equation has been treated on non-compact, geometrically non-trivial examples.
\end{abstract}

\maketitle

\tableofcontents

\normalsize


 \section{Introduction}

This article is devoted to the study of \emph{nonnegative} solutions to the Fractional Porous Medium Equation (FPME) on the hyperbolic space $\hn$. More precisely, we consider the Cauchy problem:
\begin{equation}\label{NFDE}
\left \{ \begin{array}{ll}
\partial_t u + (- \Delta_{\hn})^{s} u^m= 0\,, &  (t,x)\in  (0, \infty) \times \hn;\\
u(0,x)=u_0(x)\ge0\,, & x\in \hn,
\end{array}
\right.
\end{equation}
where  $(-\Delta_{\hn})^s$ denotes the spectral fractional Laplacian on the hyperbolic space, $0 < s < 1$,  $m > 1$, and we use the standard convention $u^m=|u|^{m-1}u$. The operator $(-\Delta_{\hn})^s$ is in fact defined by functional calculus, and well-posedness of \eqref{NFDE} in an appropriate sense will be one of the main issues in this paper, the other one being the validity of suitable \it smoothing effects \rm for such evolution, this meaning quantitative bounds on the $L^\infty$ norm of the solution at time $t>0$ in terms of a (possibly weighted) $L^p$ norm ($p\ge1$) of the initial datum. In fact, we shall prove three different estimates of that type, each dealing with a larger class of initial data and different time behaviour as $t\to0$ and $t\to\infty$ accordingly.

The study of the FPME in the Euclidean setting has been initiated in \cite{dPQRV1} for the special case $s=1/2$, and continued in the case of general exponents in \cite{dPQRV2}. In such papers well-posedness of the evolutions for $L^p$ data is proved even when $m<1$ (the \it fractional fast diffusion \rm case), and appropriate smoothing effects are proved for solutions. By a \it smoothing effect \rm we mean a bound of the form
\begin{equation}\label{smoothing}
\|u(t)\|_\infty\le C\frac{\|u_0\|_p^{\alpha} \normalcolor}{t^\beta}\ \ \ \forall t>0
\end{equation}
for suitable exponents $\alpha, \beta$, and possible generalizations of \eqref{smoothing} in which in the r.h.s. a \it weighted \rm $L^p$ norm appears. Such kind of instantaneous regularization ($L^p$ data are smoothed out instantaneously into bounded solutions) is typical of heat-like evolution equations, see e.g. \cite{DA} for the linear case and \cite{V, V2} for porous medium-type evolutions.
A number of subsequent results concerning e.g. pointwise bounds on solutions, propagation of positivity, existence, uniqueness and properties of fundamental (or Barenblatt) solution i.e. solutions corresponding to a Dirac delta as initial datum, regularity of solutions, were proved later in \cite{AC, BV3, GMP1, V3, VdPQR}.


The FPME has been later studied in bounded domains, with appropriate boundary conditions, mainly of homogeneous Dirichlet type. In fact, there are several different versions of what one might call a Dirichlet fractional Laplacian on domains and these different versions have been being actively investigated recently, first in \cite{BV2,BV1} and then e.g. in \cite{BFR, BFV, BSV}. It is particularly important in regard to the contents of the present paper that some of these papers deal with the FPME by using properties of the \it Green's function \rm associated to the version of the fractional Laplacian considered, starting from the very definition of solution, a strategy we shall use in this paper too. This is a particularly useful approach also in view of the technical difficulties in dealing with extension methods \it \`a la \rm Caffarelli-Silvestre \cite{CS} on manifolds, as developed in \cite{BGS}, in \cite{GS} for the higher rank case, see also \cite{BGFW} for the case of nonsingular kernels.

The  analysis  of the heat equation in the setting of Riemannian manifolds is a widely studied topic, see e.g.  \cite{Gri, G, G2} and references therein.
The  analysis  of \it nonlinear \rm diffusions of porous medium type on manifolds started instead just recently. The first results in this connection are to our knowledge given in \cite{BGV}, in which basic properties of the porous medium equation on Cartan-Hadamard manifolds, namely simply connected manifolds of nonpositive sectional curvature, are considered. A number of subsequent recent contributions deal with the porous and fast diffusion equations (in the nonfractional case) on manifolds, mainly in the case of negative curvature, starting with \cite{GM1,V4} and  continuing e.g. in \cite{GIM,GM,GMP2,GMP,GMV-MA,GMV}.

It should be noted that the long-time behaviour of solutions even to the \it linear \rm heat equation posed on negatively curved manifolds is completely different from the Euclidean one: in fact, one has
\begin{equation}\label{linear}
K_{\mathbb{R}^N}(t,x,x)\asymp t^{ -N/2},\ \ \ \ K_{\mathbb{H}^N}(t,x,x)\asymp t^{-3/2}e^{-(N-1)^2t}\ \ \textrm{as}\ t\to+\infty,
\end{equation}
where $K_{\mathbb{R}^N}$ is the heat kernel on the $N$-dimensional Euclidean space, and $K_{\mathbb{H}^N}$ is the heat kernel on the most important example of negatively curved, noncompact manifold, namely on the \it hyperbolic space\rm, i.e. the simply connected manifold whose sectional curvatures are constant and equal to -1.  In the series of papers just quoted above, significant differences are shown to appear when dealing with the nonlinear setting as well: in particular, the form of the smoothing effects may be quite different.\normalcolor

Little is known, however, in the fractional case. Some recent results on the FPME on compact manifolds with conical singularities are given in \cite{RS}, see also \cite{RS2, RS3} for the non-fractional case. It seems that no results are available in geometrically nontrivial, noncompact cases. Our goal here will be indeed to deal with the FPME on the hyperbolic space $\mathbb{H}^N$.

The first aim of the paper will be to introduce a concept of solution which will be based on the existence and properties of the \it fractional Green function $\GM$\rm, namely the kernel of the operator $(-\Delta_{\hn})^{-s}$. The latter operator and the corresponding kernel can be defined by functional calculus, see \eqref{g1-semigroup} below. Such definition is given precisely in Definition \ref{defi_WDS}, and involves data that are integrable w.r.t. the Riemannian measure or, at least, that are integrable \it with respect to suitable weights\rm. Such weights can be chosen to be either the \it ground state eigenfunction \rm of $-\Delta_{\hn}$, or any smooth, strictly positive function whose tail has the same decay of $\GM$ at infinity. This latter class is shown to be larger than the former, as it follows by comparing the behaviour of the ground state at infinity, see \eqref{bound}, with the one of the fractional Green function, see \eqref{Hyp.Green.HN0b}. Existence results for weak dual solutions are then given in Theorem \ref{thm-existence}, whereas the main smoothing effects are stated in the other main results of this paper, namely Theorems \ref{thm.smoothing.HN-like}, \ref{smoothPhi1}, \ref{smoothPhi}. Of course, the specific smoothing effect depends on the class of data chosen.

\medskip

The paper is organized as follows: in Section~2 we introduce some of the notations, the definition of Weak Dual Solution to \eqref{NFDE} and its relation with appropriate weighted spaces. We state our main existence and uniqueness result in Theorem~\ref{thm-existence} and then we state the smoothing estimates in
Theorems~\ref{thm.smoothing.HN-like}, \ref{smoothPhi1} and \ref{smoothPhi}. In Section~3, we derive fundamental estimates for Weak Dual Solutions
in Proposition~\ref{crucial} and we derive monotonicity in time estimates for solutions in weighted $L^1-$ spaces in
Propositions~\ref{monotonePhi}, \ref{mon_phi1}. Section~4 is devoted to the proof of the smoothing estimates stated in Section~2.
Section~5 contains existence of mild solutions using Crandall-Ligget Theory or Brezis-Komura theory and we prove that mild solutions are also
Weak Dual Solutions, see Theorem~\ref{thm-weak-mild}. Finally, we prove the existence Theorem~\ref{thm-existence}, by approximating Weak Dual Solutions (from below) in terms of mild solutions constructed in  Section 5. Appendix \ref{s-Green-function} contains a detailed proof of all crucial Green function estimates on the hyperbolic space which play a pivotal role in our study

 \section{Preliminaries and main results}
  \subsection{Definition of weak dual solutions}
  An important aspect is the right notion of solution to \eqref{NFDE}. There are many notions starting from weak solution to strong solutions.
Here we shall be dealing with the notion of \emph{weak dual solutions} introduced in \cite{BV2}. This notion has its own advantage in obtaining a priori estimates of solutions. Before formally
defining weak dual solutions, we recall some crucial facts.

The fractional power of the Laplacian on $\hn$ can be written in terms of the heat semigroup as follows:
\begin{equation*}
(- \Delta_{\hn})^{s}u = \frac{1}{\Gamma(-s)}
 \int_{0}^{\infty} ({e^{t \Delta_{\hn}}u - u}) \, \frac{{\rm d}t}{t^{1 + s}}\,,
\end{equation*}
 for an appropriate set of functions $u$   and where $\Gamma$ denotes the Euler Gamma function. On the other hand, we denote by $\GM(x,x_0)$ the fractional Green function on $\hn$ with fixed pole $x_0\in \hn$, and by $\mu_{\hn}$ its Riemannian measure. Functional calculus allows to define the inverse of $(- \Delta_{\hn})^{s}$,  on an appropriate set of functions $u$, as
\begin{equation}\label{g1-semigroup}
(- \Delta_{\hn})^{-s}u = \int_0^{+\infty} \frac{e^{t \Delta_{\hn}}u}{t^{1-s}}\,dt.
\end{equation}
Since the kernel of $(- \Delta_{\hn})^{-s}$ is $\GM$, there holds:
\begin{equation*}
(- \Delta_{\hn})^{-s} u=\int_{\hn} \GM(\cdot, y) \,u(y)\, \dg(y).
\end{equation*}

In the following we will deal with solutions with initial data in $L^1(\hn)$ or in suitable weighted spaces containing $L^1(\hn)$. More precisely, we will deal with two kinds of weights, described below. The smoothing effects proved in the two classes will of course be different.


\medskip

The first class of weights we will deal with is defined as follows:
\begin{equation}\label{W}
 \mathcal{W} =
 \left \{\Phi= (-\Delta_{\hn})^{-s}\psi \quad \text{for some nonnegative } \psi\in \mbox{C}_c^{\infty}(\hn) \right\}\,.
\end{equation}
It's worth noting that $\mathcal{W} \subset L^{\infty}(\hn)$ and weights in $\mathcal{W}$ behave asymptotically like the Green function at infinity. More precisely, in Lemma \ref{Green-operator-estimate} below, we prove that for all $\Phi \in \mathcal{W}$ there exist
$R,C_1, C_2>0$ such that
\begin{equation}\label{boundgreen}
 C_1\,r^{s-1}\,e^{-(N-1) r} \leq  \Phi (x) \leq C_2\,r^{s-1}\,  e^{-(N-1) r}  \quad  \text{ for } r\geq  R\,
\end{equation}
where  $r:=r(x,o)$ denotes the geodesic distance from a fixed $o \in \hn$.

To introduce the second weight, we recall the definition of \emph{generalised ground state} for $(- \Delta_{\hn})^{s}$. Let $\Lambda = \frac{(N-1)^2}{4}$ be the bottom of the spectrum of the Laplace-Beltrami operator on $\hn$. Then the equation
$$
(- \Delta_{\hn})^{s} u = \Lambda^s u \quad \text{in }\hn
$$
admits a positive radial solution  $\Phi_1(x)=\Phi_1(r(x,o))$ . $\Phi_1$ is then called a generalised ground state or the eigenfunction, and is defined up to multiplicative constants and up to the choice of the pole $x_0$. It is known e.g. from \cite[Section 3.3]{BGGV}, that $\Phi_1$ satisfies
\begin{equation}\label{bound}
\alpha (1 + r)e^{-\frac{(N-1)}{2}r}  \leq \Phi_1(x) \leq \beta (1 + r)e^{-\frac{(N-1)}{2}r} \quad \text{for all } r > 0
\end{equation}
and for some $\alpha,\beta>0$.\par  For $\Phi \in \mathcal{W}$ or $\Phi=\Phi_1$, we denote by
$$L^{1}_{\Phi} (\hn):=\{u: \hn \rightarrow \hn \text{ measurable } : \sup_{x_0\in \hn} \int_{\hn} \, |u(x)|\,(\Phi \circ \tau_{-x_0})(x)\, \dgh(x) < +\infty\}$$ and by $\|u\|_{L^1_{\Phi}(\hn)}:=\sup_{x_0 \in \hn} \int_{\hn} \, |u(x)|\,(\Phi \circ \tau_{-x_0})(x)\, \dgh(x)$, where $\tau_{-x_0}$ is the standard hyperbolic translation.  Note that \eqref{boundgreen} and \eqref{bound} readily imply that $\Phi\le C\Phi_1$ on $\hn$, for a suitable $C>0$, so that $L^{1}_{\Phi_1} (\hn)\subset L^{1}_{\Phi} (\hn)$ for all $\Phi \in \mathcal{W}$. In particular, since $\Phi=o(\Phi_1)$ at infinity, the inclusion is strict.

We are finally ready to state
\begin{defi}\label{defi_WDS}
Let $\Phi \in \mathcal{W}$, where $\mathcal{W}$ is given in \eqref{W}, or $\Phi=\Phi_1$, a ground state of $-\Delta_{\hn},$  or $\Phi\equiv 1$. We say that $u$ is a  Weak Dual Solution (WDS) to Problem~\eqref{NFDE} in
$[0, T) \times \hn$ if

\begin{itemize}
\item  $u \in C([0, T) ; L^{1}_{\Phi \circ \tau_{-x_0}} (\hn))$, for all $x_0 \in \hn$   and $u^m \in L^1\left( (0, T) ; L^{1}_{loc}(\hn) \right)$ ;

\medskip

\item $u$ satisfies the following identity

\begin{equation}\label{def_eq}
\int_{0}^{T} \int_{\hn} (- \Delta_{\hn})^{-s}(u) \, \partial_t \psi \,
\dg \, {\rm d}t - \int_{0}^{T} \int_{\hn} u^m \psi \, \dg\, {\rm d}t = 0
\end{equation}
for every test function $\psi \in C^1_c(0,T; \mbox{L}_c^{\infty}(\hn))$ ;

\medskip

\item $u(0,x)=u_0(x)$ a.e. in $\hn$.

\end{itemize}
\end{defi}
%
%
%
\begin{rem}\label{remdef}
We note that the first term of \eqref{def_eq} is finite. Indeed, using Fubini Theorem, we have
$$
\int_{0}^{T} \int_{\hn} (- \Delta_{\hn})^{-s}(u) \, \partial_t\psi \,
\dg \, {\rm d}t = \int_{0}^{T} \int_{\hn} (- \Delta_{\hn})^{-s}\left(\partial_t\psi \right) \, u  \,
\dg \, {\rm d}t
$$
for all $\psi \in C^1_c(0,T; \mbox{L}_c^{\infty}(\hn))$. Hence, if   $u \in C([0, T) ; L^{1}_{\Phi \circ \tau_{-x_0}} (\hn))$, for $\Phi  \in \mathcal{W}$ and for all $x_0 \in \hn$,   then $\Phi=(-\Delta_{\hn})^{-s}\overline \psi  \in \mathcal{W}$ for some positive $ \overline \psi \in C^{\infty}_c(\hn)$ and, since $\frac{(-\Delta_{\hn})^{-s} \left(\partial_t \psi\right) }{(-\Delta_{\hn})^{-s}\overline \psi }\in L^{\infty}(\hn)$  (see Lemma \ref{Green-operator-estimate} below) we conclude that the above integrals are finite.

Since $L^1(\hn)\subset  L^{1}_{\Phi_1} (\hn)\subset L^{1}_{\Phi} (\hn)$ for all $\Phi \in \mathcal{W}$, the same conclusion holds if either $\Phi=\Phi_1$ or $\Phi\equiv 1$.

This class of solutions is quite big, indeed it contains mild (semigroup) $L^1$-solutions, see Section \ref{sect.semigroups}. On the other hand, in Remark \ref{Lp-stab-WDS} we see that WDS are contained in the class of very weak (distributional) solutions.\normalcolor
\end{rem}
\subsection{Main results}
The construction of a WDS will be done by means of approximating with semigroup (mild) solutions constructed either in $L^1(\hn)$ or in $H^{-s}(\hn)$, by means of appropriate nonlinear semigroup techniques both in Banach and Hilbert spaces. In Section \ref{sect.semigroups}  we show that mild solutions are weak dual solutions. 

\begin{thm}[Existence and uniqueness of  Nonnegative Minimal Weak Dual Solutions in $L^1_{\Phi}$]\label{thm-existence} Let $\Phi \in \mathcal{W}$, where $\mathcal{W}$ is given in \eqref{W}, or $\Phi=\Phi_1$, a  ground state of $-\Delta_{\hn}$, or $\Phi\equiv1$. For every  nonnegative $ u_0\in L^1_{\Phi}(\Omega)$ there exists a weak dual solution to Problem~\eqref{NFDE}.   Such a solution is obtained as the unique monotone limit of nonnegative $L^1$-semigroup (mild) solutions which exist and are unique. We call such solution the \textsl{minimal WDS}. \normalcolor In this class of solutions the standard comparison principle holds.
%
\end{thm}
\begin{rem}\textit{i}) \textit{Existence and uniqueness of signed (semigroup) solutions. }
\rm
When dealing with signed solutions, existence and uniqueness of mild (hence WDS) solutions follows from the nowadays standard theory of nonlinear contractive semigroups on Banach and Hilbert spaces. In the first case, the $L^1(\hn)$-theory of m-Accretive operators, developed by Benilan, Crandall, Pazy and Pierre provides existence, uniqueness and comparison, see \cite{BCr,BCPbook,CP} and Theorem \ref{BCPP-Thm}. In the Hilbertian setting, the (nowadays called Gradient Flow) theory by Brezis and Komura, gives the same results in  $H^{-s}(\hn)$, see \cite{Brezis1,BrezisBk,Komura} and Theorem \ref{BK-Thm}. More details will be provided in Section \ref{gradient}. \normalcolor

\noindent\textit{ii})\noindent\textit{Boundedness of solutions. }On one hand, showing that mild solutions are bounded is in general a difficult task, unless the initial datum is already bounded, since the a-priori regularity is too low. On the other hand, WDS turn out to be bounded, as we will see below. We show in Section \ref{sect.semigroups} that mild solutions are indeed WDS, hence bounded. Informally speaking, using WDS means to find appropriate weak solutions to the dual equation $(-\Delta_{\hn})^{-s}u_t = -u^m$: this is useful since it allows to extend the approach via Green functions estimates of \cite{BV1,BV2} to the present setting, and to prove quantitative boundedness of WDS, avoiding DeGiorgi-Nash-Moser iterations, which rely on integration by parts, not available in this non-local framework.

 \noindent\textit{iii})\noindent\textit{Initial data when $\Phi \neq 1$.} It's worth pointing out that, when $\Phi \neq 1$, the considered class of initial data is larger than $L^1(\hn)$. As a matter of example, consider the function $u_{\gamma}(r)=\frac{e^{-r(N-1)}}{r^{\gamma}}$ with $0 \leq \gamma <1$, where $r:=r(x,o)$ for a fixed $o \in \hn$. Clearly, $u_{\gamma} \not \in L^1(\hn)$ while $u_{\gamma}\in L^1_{\Phi}(\hn)$ for $\Phi=\Phi_1$ and therefore also for $\Phi \in \mathcal{W}$. To see this, since $u_{\gamma}$ is locally integrable and, by \eqref{bound}, $\Phi_1(x)\leq \beta$ for all $x \in \hn$, we only need to check that:
\begin{align}\label{wa-k}
\sup_{x_0\in \hn}  \int_{B_1^c(o)} |u_{\gamma}(x)|\,  (\Phi_1 \circ \tau_{-x_0})(x) \, \dgh(x) < +\infty  \, .
\end{align}
By \eqref{bound}, $(\Phi_1 \circ \tau_{-x_0})(x)\leq C e^{-\frac{(N-1) r(x, x_0)}{4}}$ for some $C>0$ independent from $x_0$, and we estimate
\begin{align*}
&\int_{B_1^c(o)} |u_{\gamma}(x)|\,  (\Phi_1 \circ \tau_{-x_0})(x) \, \dgh(x) \\
&\leq   \int_{ B_1^c(o) \cap B_{r(x_0,o)}(o)}  |u_{\gamma}(x)|\, C e^{-\frac{(N-1) [r(x_0,o)-r(x,o)]}{4}} \, \dgh(x)\\
&+   \int_{ B_1^c(o) \cap B_{r(x_0,o)}^c(o)}  |u_{\gamma}(x)|\, C e^{-\frac{(N-1) [r(o,x)-r(x_0,o)]}{4}} \, \dgh(x)\,.
\end{align*}
Taking $r(x_0,o)>1$ in the above, the other case being analogous, we get
\begin{align*}
&\int_{B_1^c(o)} |u_{\gamma}(x)|\,  (\Phi_1 \circ \tau_{-x_0})(x) \, \dgh(x) \\
&\leq   \frac{C}{2} e^{-\frac{(N-1) r(x_0,o)}{4}}  \int_1^{r(x_0,o)} \frac{e^{\frac{(N-1) r}{4}}}{r^{\gamma}}\,dr+ \frac{C}{2} e^{\frac{(N-1) r(x_0,o)}{4}}  \int_{r(x_0,o)}^{+\infty} \frac{e^{-\frac{(N-1) r}{4}}}{r^{\gamma}}\,dr \leq \frac{4C}{N-1}
\end{align*}
and \eqref{wa-k} follows.

We notice that the weights $\Phi \in \mathcal{W}$ allow for even larger data than $u_{\gamma}$. As a matter of example, the function $w_{\gamma}(r)=\frac{1}{r^{\gamma}}$ with $s < \gamma <1$, satisfies $w_{\gamma}\in L^1_{\Phi}(\hn)$ but $w_{\gamma}\not\in L^1_{\Phi_1}(\hn)$. We refer to \cite[Section 4]{BBGG} for a proof and for a discussion about the set $L^1_{\Phi}$ on general noncompact manifolds.   

\end{rem}

It is convenient to define the exponent
\begin{equation*}
\vartheta_1:=\frac{1}{2s+N(m-1)}\,
\end{equation*}
which appears in our $L^1-L^{\infty}$ smoothing estimates.

\begin{thm}\label{thm.smoothing.HN-like}
Let $u $ be   a WDS  corresponding to the nonnegative initial datum $u_0\in L^1(\hn)$. There exists $\ka_1=\ka_1(s,m,N)>0$ such that
\begin{equation}\label{thm.smoothing.HN-like.estimate}
\|u(t)\|_{L^\infty(\hn)} \leq  \ka_1\frac{\|u(t)\|^{2s \vartheta_1}_{L^1(\hn)}}{t^{N \vartheta_1}}\leq  \ka_1\frac{\|u_0\|^{2s \vartheta_1}_{L^1(\hn)}}{t^{N \vartheta_1}}\qquad \mbox{for all }t>0.
\end{equation}
Moreover, for $t\ge  e^{2(N-1)(m-1)}\|u_0\|_{L^1(\hn)}^{-(m-1)}$ there exists $\ka_2=\ka_2(s,m,N)>0$ such that
\begin{equation}\label{thm.smoothing.HN-like.estimate.2}
\|u(t)\|_{L^\infty(\hn)}  \leq \frac{\ka_2 }{t^{\frac{1}{m-1}}}\left[ \log\left(t \,\|u_0\|_{L^1(\hn)}^{m-1}\right)\right]^{\frac{s}{m-1}}  \,.
\end{equation}
\end{thm}
%

A second smoothing effect can be shown in the larger class, namely for data belonging to $L^1_{\Phi_1}(\hn)$.

\begin{thm}\label{smoothPhi1}
Let $\Phi_1$ be a ground state of $-\Delta_{\hn}$ and let $u $ be a WDS   corresponding to the nonnegative initial datum $u_0\in L^1_{\Phi_1}(\hn)$.
There exists $\ka_3=\ka_3(s,m,N)>0$ such that
\begin{equation}\label{eq-k3}
\|u(t)\|_{L^\infty(\hn)} \leq  \ka_3\frac{\|u(t)\|^{2s \vartheta_1}_{L^1_{\Phi_1}(\hn)}}{t^{N \vartheta_1}}\leq  \ka_3\frac{\|u_0\|^{2s \vartheta_1}_{L^1_{\Phi_1}(\hn)}}{t^{N \vartheta_1}}\qquad \mbox{for all }t>0\,.
\end{equation}

Moreover, for $t\ge  e^{(m-1)(N-1)}\|u_0\|_{L^1_{\Phi_1}(\hn)}^{-(m-1)}$ there exists $\ka_4=\ka_4(s,m,N)>0$ such that
\begin{equation}\label{eq-k5}
\|u(t)\|_{L^{\infty}(\hn)} \leq \frac{\ka_5}{t^{\frac{1}{m-1}}} \, \left[ \log( t\, \|u_0\|_{L^1_{\Phi_1}(\hn)}^{m-1})\right]^{\frac{s}{m-1}}\,.
\end{equation}
\end{thm}
\begin{rem} It should be commented that, in the above Theorems, the bounds \eqref{thm.smoothing.HN-like.estimate}, \eqref{eq-k3} hold for all $t>0$ but are mainly significant for $t$ small, as for long time faster decay estimates are provided by \eqref{thm.smoothing.HN-like.estimate.2}, \eqref{eq-k5}. This is reminiscent of the situation valid in Euclidean \it bounded \rm domains, in which short and long time asymptotics of solutions to the porous medium equation (or even to the heat equation), fractional or not, are different, mainly due to the validity of the Poincar\'e inequality, a fact that holds on $\hn$ as well. Note that \eqref{thm.smoothing.HN-like.estimate} is identical to the corresponding Euclidean bound and is consistent with the Euclidean scaling. The bound \eqref{eq-k3} has the same time dependence but, being written in terms of a weighted norm, has no Euclidean analogue. That for large time a faster time decay overtakes the one valid in the Euclidean case might also informally be seen as an effect of negative curvature, that somehow increases the speed of propagation and produces a better decay for large time (recall the bounds \eqref{linear} in the linear case).
\end{rem}
We conclude by enlarging further the class of allowed initial data, i.e. by allowing them to be integrable w.r.t. weight in $\mathcal{W}$, namely having the tail of the fractional Green function.
\begin{thm}\label{smoothPhi}
Let $u $ be a WDS  corresponding to the nonnegative initial datum  $u_0\in L_{\Phi}^1(\hn)$, with $\Phi \in \mathcal{W}$ defined \normalcolor in \eqref{W}. There exist $\ka_5=\ka_5(s,m,N)>0$ such that 

\begin{equation}\label{eq-k1PHI1}
\|u(t)\|_{L^\infty(\hn)} \leq  \ka_5\frac{\|u(t)\|^{2s \vartheta_1}_{L^1_{\Phi}(\hn)}}{t^{N \vartheta_1}}\leq  \ka_5\frac{\|u_0\|^{2s \vartheta_1}_{L^1_{\Phi}(\hn)}}{t^{N \vartheta_1}}\qquad \mbox{for all } t\leq \|u_0\|_{ L^1_{\Phi}(\hn)}^{1-m}. 
\end{equation}
Moreover, there exists $\ka_6=\ka_6(s,m,N, \Phi)>0$ such that  for all $t  \geq \|u_0\|_{ L^1_{\Phi}(\hn)}^{1-m}$, there holds
\begin{equation}\label{eq-k1PHI2}
\|u(t)\|_{L^{\infty}(\hn)} \leq  \frac{\ka_6}{t^{1/m}}  \|u_0\|_{  L^1_{\Phi}(\hn)}^{1/m} \,.
\end{equation}
\end{thm}
\begin{rem}\label{rem.smooth.L1.H-s}\textit{About the space $L^1_{\Phi}(\hn)$ versus $H^{-s}(\hn)$. }The above smoothing effect shows that nonnegative WDS corresponding to data in $L_{\Phi}^1(\hn)$ are bounded. As we shall see in Remark \ref{rem.strong.H-s}, the gradient flow solutions with data in $H^{-s}(\hn)$, which exist and are unique, are also WDS, hence the above Theorem applies also to nonnegative gradient flow solutions and gives  for all $t\geq \|u_0\|_{ L^1_{\Phi}(\hn)}^{1-m}$: 
\begin{equation*}
\|u(t)\|_{L^{\infty}(\hn)} \leq  \frac{\ka_6}{t^{1/m}} \|u_0\|_{  L^1_{\Phi}(\hn)}^{1/m}
\leq  \frac{\ka_6\,c_\psi^{1/m}}{t^{1/m}} \|u_0\|_{H^{-s}(\hn)}^{1/m}\,,
\end{equation*}
where we have used inequality \eqref{L1-H-s.norms}, i.e. $\|u_0\|_{  L^1_{\Phi}(\hn)}\le c_\psi \|u_0\|_{H^{-s}(\hn)}$ with $c_\psi =\|\psi\|_{L^1_{\Phi}}<+\infty$. The latter inequality also shows that for nonnegative functions, $H^{-s}_+$ is contained in $L^1_{\Phi,+}(\hn)$.
\end{rem}

\section{Fundamental estimates for weak dual solutions (WDS)}

In this section we derive a priori estimates
of WDS in the spirit of \cite[Sections 5 and 6]{BV1} and \cite{BV2}. To this end, we will use  repeatedly  the next lemma.
%
%
%
%
%
\begin{lem}\label{lem.Green.HN.est}
There exists $\kappa_1=\kappa_1(N,s)>0$ such that for a.e. $x,y\in \hn$ we have
\begin{equation}\label{Hyp.Green.HN0}
 \GH(x,y)\le \kappa_1\frac{ 1}{r(x,y)^{N-2s}}
\end{equation}
where $r(\cdot, \cdot)$ is the geodesic distance of $\HH^N$. When $r(x,y)\ge 1$ (eventually by taking a bigger $\kappa_1$) we also have
\begin{equation}\label{Hyp.Green.HN0b}
\GH(x,y)\le \kappa_1 \frac{ e^{-(N-1) r(x, y)}}{r(x, y)^{1-s}}\,.
\end{equation}
As a consequence, for all $R>0$  and $y\in \hn$ we have
\begin{equation}\label{Hyp.Green.HN1}
\int_{B_R(y)}\GM(x,y)\dg(x)
\le  \overline{\overline{\kappa}}_1  R^{2s},
\end{equation}
and for all $R\ge 2$ and $y\in \hn$ we have
\begin{equation}\label{Hyp.Green.HN2}
\int_{B_R(y)}\GM(x,y)\dg(x)
\le \widetilde \kappa_1 R^s,
\end{equation}
 for some $ \overline{\overline{\kappa}}_1= \overline{\overline{\kappa}}_1(N,s, \kappa_1)>0$ and  $\widetilde \kappa_1=\widetilde \kappa_1(N,s, \kappa_1)>0$ .
\end{lem}\normalcolor

\begin{proof}In order not to break the flow of the section we will pospone the (technical) proofs of \eqref{Hyp.Green.HN0} and \eqref{Hyp.Green.HN0b} to Appendix \ref{s-Green-function}, see Corollary \ref{cor-1}.   First we prove \eqref{Hyp.Green.HN1} for $0<R\leq 2$;  it follows from \eqref{Hyp.Green.HN0} by passing to radial coordinates and estimating the volume element by means of the elementary inequality:  $\sinh r\leq r\cosh \,R$ for all $0<r<R$. Indeed

\begin{align*}
\int_{B_R(y)}\GM(x,y)\dg(x) & \leq \kappa_1 \int_{0}^{R} \frac{(\sinh r)^{N-1}}{r^{N-2s}} \, {\rm d}r
= \kappa_1 \int_{0}^{R} \frac{r^{N-1}}{r^{N-2s}} \left(\frac{\sinh r}{r} \right)^{N-1}  \, {\rm d}r \\
& \leq \frac{\kappa_1}{2s}  (\cosh 2)^{N-1}  R^{2s}.
\end{align*}

As for \eqref{Hyp.Green.HN2} for $R\geq 2$, we estimate using the fact that $e^{-(N-1)r} (\sinh r)^{N-1} \leq 1,$
$$\int_{B_R(y)}\GM(x,y)\dgh(x)\le   \frac{\kappa_1}{2s} \cosh(2)^{N-1} 2^{2s}+ \kappa_1 \int_{2}^R\,\frac{ e^{-(N-1) r}}{r^{1-s}} \, (\sinh r)^{N-1}\, dr$$
$$\le \frac{\kappa_1}{s} \left[ \cosh(2)^{N-1} 2^{2s-1}-2^s +R^s \right]\leq \widetilde \kappa_1 \, R^s$$
for some $\widetilde \kappa_1=\widetilde \kappa_1(N,s, \kappa_1)>0$ .

 Finally, the proof of \eqref{Hyp.Green.HN1} for all $R>0$ follows by combining \eqref{Hyp.Green.HN1} for $0<R\leq 2$ and \eqref{Hyp.Green.HN2} for $R\geq 2$.  
 \end{proof}
We are now ready to state and prove what we call the fundamental pointwise formulae, that in the nonlinear case under study, play the role of the representation formula in the linear case.

Following the strategy of \cite{BV2}, it is convenient at this point to introduce a special class of WDS, useful to ensure that all the integral quantities in the proofs are finite. Indeed, in most of the proofs we will work with bounded integral solutions and then extend the result to general data by a standard limiting  process. To this end, we shall proceed initially by considering WDS having initial data in $L^1(\hn)\cap L^\infty(\hn)$.

%
%
%

%
%
\par
%
%

\begin{lem}[Comparison and time monotonicity for WDS]\label{WDS.comp+monot}
Let $u,v$ be WDS corresponding to $u_0, v_0\in L^1(\hn)\cap L^\infty(\hn)$. Then comparison holds, more precisely  the T-contraction property \eqref{T-contraction.2} holds, namely
\begin{equation}\label{comp.H-s}
\left\|\big(u(t)-v(t)\big)_+\right\|_{H^{-s}(\hn)}\le \left\|\big(u_0-v_0\big)_+\right\|_{H^{-s}(\hn)}\qquad\mbox{for all }t\ge 0\,.
\end{equation}
Let $u$ be a nonnegative WDS corresponding to $u_0\in L^1(\hn)\cap L^\infty(\hn)$. Then $u$ enjoys the following time monotonicity property:
\begin{equation}\label{mon-est}
\mbox{the map}\qquad t \mapsto t^{\frac{1}{m-1}} u(t,x) \qquad\mbox{is nondecreasing in $t > 0$ for a.e. $x \in \hn,$}
\end{equation}
and the following $L^p-$ stability property
\begin{equation}\label{eq-wm-12}
\|u(t)\|_{L^p(\hn)} \leq \|u_0\|_{L^p(\hn)} \qquad\mbox{for all $t\ge 0$ and all $1 \le p \leq \infty.$}
\end{equation}

\end{lem}
%
%
%
%
%
\noindent {\bf Proof.~}This proof follows by the T-contraction in $H^{-s}$-spaces, that we adapt from an original proof of Brezis, we just emphasize the main key points, the rest being completely analogous to the $H^{-s}$-case, see more details in Section \ref{gradient}.
%
By the Poincar\'e inequality $\lambda_1\|f\|_{H^{-s}(\hn)}\le \|f\|_{L^2(\hn)}$. Then,  since $u_0, v_0\in L^1(\hn)\cap L^\infty(\hn)$, we have that $(u_0-v_0)_+ \in H^{-s}(\hn)$ and the following holds
\[\begin{split}
\frac{d}{dt}\frac{1}{2}&\left\|\big(u(t)-v(t)\big)_+\right\|_{H^{-s}(\hn)}^2\\
&=\int_{\hn} {\rm sign}^{+}\big(u(t)-v(t)\big)\,\big(\partial_tu(t)-\partial_tv(t)\big)\, (-\Delta_{\hn})^{-s}\big(u(t)-v(t)\big)_+ \dg\\
&=-\int_{\hn} {\rm sign}^{+}\big(u(t)-v(t)\big)\,\left[(-\Delta_{\hn})^{s}\big(u^m -v^m\big)  \right]\, (-\Delta_{\hn})^{-s}\big(u(t)-v(t)\big)_+ \dg\\
&=-\int_{\hn} {\rm sign}^{+}\big(u^m(t)-v^m(t)\big)\,\left[(-\Delta_{\hn})^{s}\big(u^m -v^m\big)  \right]\, (-\Delta_{\hn})^{-s}\big(u(t)-v(t)\big)_+ \dg\\
&\le-\int_{\hn} (-\Delta_{\hn})^{s}\big(u^m -v^m\big)_+   (-\Delta_{\hn})^{-s}\big(u(t)-v(t)\big)_+ \dg\\
&\le-\int_{\hn}  \big(u^m -v^m\big)_+    \big(u(t)-v(t)\big)_+ \dg\le 0
\end{split}
\]
which clearly implies \eqref{comp.H-s}. The inequality of line 4 follows by Kato's inequality ${\rm sign}^{+}\big(f\big)(-\Delta_{\hn})^{s}f\le (-\Delta_{\hn})^{s} (f)_+ $, which holds in the sense of distributions.

The above computations can be rigorously justified by the results of Theorem \ref{BK-Thm}, in which strong $H^{-s}$-solutions are constructed. When $u_0$ is sufficiently integrable (and $u_0\in L^1(\hn)\cap L^\infty(\hn)$ is more than enough in view of the above discussion) then the strong $H^{-s}$ solutions are the same as Weak Dual Solutions in the sense of Definition \ref{defi_WDS}, see Remark \ref{rem.strong.H-s}.

\noindent\textit{Time Monotonicity Estimates. }We show here for the reader's convenience a proof based on scaling and comparison that we have learned from V\'azquez \cite{V}. Consider the rescaled solution $u_\lambda(t,x)=\lambda^{\frac{1}{m-1}}u(\lambda t,x)$, then $u_\lambda(0,\cdot)=\lambda^{\frac{1}{m-1}}u_0$. Then, letting $\lambda=(t+h)/t\ge 1$, we obtain for all $h\ge 0$
\[
0\le  u_\lambda(t,x)-u(t,x) = \lambda^{\frac{1}{m-1}}u(\lambda t,x)-u(t,x)=\left( \frac{t+h}{t}\right)^{\frac{1}{m-1}}u(t+h,x)-u(t,x)
\]
which is equivalent to the desired monotonicity of the map $t \mapsto t^{\frac{1}{m-1}} u(t,x)$. The first inequality in the above formula, is true for all $t\ge0$ and a.e. $x\in \hn$ by comparison (which follows by the T-contraction inequality \eqref{comp.H-s}), since $\lambda \ge 1$ implies that $u_\lambda(0,\cdot)=\lambda^{\frac{1}{m-1}}u_0\ge u_0$.

\noindent\textit{$L^p$ stability of nonnegative WDS with $u_0\in L^1(\hn)\cap L^\infty(\hn)$. }A sketch of the proof of inequality \eqref{eq-wm-12} is given in Remark \ref{Lp-stab-WDS}. \qed

We are now ready to prove our main result of this section:\normalcolor

\begin{prop}\label{crucial}Let $u \geq 0$ be a WDS to problem \eqref{NFDE} with $u_0\in L^1(\hn)\cap L^\infty(\hn)$. Then,\normalcolor
\begin{equation}\label{estimate-1}
\int_{\hn} u(t, x)  \, \GM (x, x_0) \, \dg(x)
 \leq \int_{\hn} u_0(x) \,  \GM(x, x_0) \, \dg(x) \text{ for all } t>0
\end{equation}
and
\begin{align}\label{main-estimate-1}
 \left( \frac{t_0}{t_1} \right)^{\frac{m}{m-1}} (t_1 - t_0)\, u^m (t_0, x_0)
 &\leq
  \int_{\hn} \big[u(t_0, x)-u(t_1, x)\big] \, \GM(x_0, x) \, \dg(x)
  \notag  \\ &\leq (m-1) \frac{t^{\frac{m}{m-1}}}{t_0^{\frac{1}{m-1}}} \, u^m(t, x_0)\,,
\end{align}
for a. e. $0<t_0\leq t_1\leq t$ and a. e. $x_0\in \hn$.
\end{prop}

\begin{proof}
We give the proof for $\Phi \in \mathcal{W}$. The case $\Phi=\Phi_1$ follows by noting that $L^{1}_{\Phi_1} (\hn)\subset L^{1}_{\Phi} (\hn)$ for all $\Phi \in \mathcal{W}$. The statement is proved in four steps.
\medskip

\noindent{\bf Step 1 :} For all non-negative $\psi$ such that  $\psi \in L^{\infty}_{c}(\hn)$, there holds
\begin{align}\label{limit-1}
&\int_{\hn} u(t_0, x) (- \Delta_{\hn})^{-s} (\psi(x)) \, \dg(x)
- \int_{\hn} u(t_1, x) (- \Delta_{\hn})^{-s} (\psi(x)) \, \dg(x)  \notag \\
=&  \int_{t_0}^{t_1} \int_{\hn} u^{m}(\tau, x) \psi(x) \, \dg(x) {\rm d}\tau\,.
\end{align}
The proof follows by adapting the ideas of \cite[Proposition~4.2-Step 1]{BV2}. We use the Definition~\ref{defi_WDS} of weak dual solution, with test function
 $\psi(t, x) = \psi_1(t) \psi_2(x),$ where $\psi_1(t) \in C^1_c(0, +\infty)$ and $\psi_2 \in L^{\infty}_{c} (\hn).$ Namely, we have that $u \in C([0, T) : L^{1}_{\Phi} (\hn))$, $u^m \in L^1\left( (0, T) : L^{1}_{loc}(\hn) \right)$, and the following identity holds
 \begin{align}\label{imp-identity}
& \int_0^{\infty} \psi_{1}^{\prime}(\tau) \int_{\hn} u(\tau, x) \, (-\Delta_{\hn})^{-s} \psi_2(x) \, \dg(x) \, {\rm d}\tau \notag \\
 &= \int_0^{\infty} \psi_{1}(\tau) \int_{\hn} \, u^m(\tau, x) \, \psi_2(x) \, \dg(x) \, {\rm d}\tau.
 \end{align}
 Now, to prove \eqref{limit-1}, we exploit the following approximate procedure. For  $0 \leq t_0 \leq t_1 < +\infty$ we consider the characteristic function $\chi_{[t_0, t_1]} (\tau)$. By standard approximation arguments, there exists $\psi_{1, n} \in C_c^{1}(0, + \infty)$ with ${\rm supp} (\psi_{1, n}) \subset \left[ t_0 - \frac{1}{n}, t_1 + \frac{1}{n} \right]$, $\chi_{[t_0, t_1]} \leq \psi_{1,n} \leq 1$ and such that $\psi_{1, n}  \rightarrow \chi_{[t_0, t_1]} (\tau) \text{ a.e. in } (0,+\infty)$. Since, by Remark \ref{remdef}, $\tau \mapsto \int_{\hn} u(\tau, x) \, (-\Delta_{\hn})^{-s} \psi_2(x) \, \dg(x) \in C^0([0,T))$, as $n \rightarrow \infty$, it follows that
\begin{align*}
&\int_0^{\infty} \psi_{1,n}^{\prime}(\tau) \int_{\hn} u(\tau, x) \, (-\Delta_{\hn})^{-s} \psi_2(x) \, \dg(x) \, {\rm d}\tau \\
&\xrightarrow[n \to \infty]{}\, \int_{\hn} \, u(t_0, x) (-\Delta_{\hn})^{-s} \psi_2(x) \, \dg(x)  - \int_{\hn} \, u(t_1, x) (-\Delta_{\hn})^{-s} \psi_2(x) \, \dg(x).
\end{align*}
 On the other hand, since
 $$\int_0^{\infty} \psi_{1,n}(\tau) \int_{\hn} \, u^m(\tau, x) \, \psi_2(x) \, \dg(x) \, {\rm d}\tau \leq \int_0^{t_1+1} \int_{\hn} \, u^m(\tau, x) \, \psi_2(x) \, \dg(x) \, {\rm d}\tau <+\infty$$ by Lebesgue dominated Theorem, as $n \rightarrow \infty$, we get:
 \begin{align*}
 \int_0^{\infty} \psi_{1,n}(\tau) \int_{\hn} \, u^m(\tau, x) \, \psi_2(x) \, \dg(x) \, {\rm d}\tau
 & \xrightarrow[n \to \infty]{} \, \int_{t_0}^{t_1} \int_{\hn} \,
 u^m(\tau, x) \, \psi_2(x) \, \dg(x) \, {\rm d}\tau\,.
 \end{align*}
Hence, \eqref{limit-1} follows by writing \eqref{imp-identity} with $\psi_1=\psi_{1,n}$ and passing to the limit.

\medskip

\noindent{\bf Step 2 :} From \eqref{limit-1} we derive estimate \eqref{estimate-1}.
Fix $x_0 \in \hn$ and consider the sequence of nonnegative test functions
$\psi_{n}^{(x_0)}(x):=\frac{1}{|B_{1/n}(x_0)|}\,\chi_{B_{1/n}(x_0)}(x)$ with $n>1$. Clearly,  $\psi^{(x_0)}_{n} \in L^{\infty}_{c}(\hn)$. Furthermore, $\psi_{n}^{(x_0)} \xrightarrow[n \to \infty]{}\delta_{x_0}$ in the sense of Radon measure.

\medskip

A direct application of Lebesgue differentiation theorem to the function $y\mapsto \GM(y, x)$ (without loss of generality we may
assume that $x_0$ is a Lebesgue point) implies
$$
(- \Delta_{\hn})^{-s}(\psi_{n}^{(x_0)})(x)=\frac{1}{|B_{1/n}(x_0)|} \int_{B_{1/n}(x_0)} \GM(y, x)\,\dg(y)\,
\xrightarrow[n \to \infty]{} \, \GM(x,x_0)
$$
as $n \rightarrow \infty$ and for a.e. $x\in \hn$.  Next we write
\begin{align}\label{step2lim}
&\left| \int_{\hn} u(\tau, x) (- \Delta_{\hn})^{-s} \psi_{n}^{(x_0)}(x) \,\dg(x)
-  \int_{\hn} u(\tau, x) \GM(x, x_0) \, \dg(x)\right| \\  \notag
& \leq  \left| \int_{B_R(x_0)} u(\tau, x)(- \Delta_{\hn})^{-s} \psi_{n}^{(x_0)}(x) \, \dg(x)
-  \int_{B_R(x_0)} u(\tau, x) \GM(x, x_0) \, \dg(x) \, \right| \\ \notag
& + \left| \, \int_{\hn \setminus B_R(x_0)}
u(\tau, x) (- \Delta_{\hn})^{-s} \psi_{n}^{(x_0)}(x) \, \dg(x)
-  \int_{\hn\setminus B_R(x_0)} u(\tau, x) \GM(x, x_0) \, \dg(x) \right|\,.
\end{align}
By Lemma \ref{lem.Green.HN.est}, $\GM( x_0, \cdot) \in L^q_{loc}(\hn)$ with $0 < q < \frac{N}{N -2s}$ and we
readily derive that
\begin{equation*}
(- \Delta_{\hn})^{-s}\psi_{n}^{(x_0)}(x) \xrightarrow[n \to \infty]{} \GM( x_0, x)
 \quad \mbox{in} \ L^{q}_{loc}(\hn)
\end{equation*}
for all $0 < q < \frac{N}{N -2s}$. Then, the first expression on the r.h.s. of \eqref{step2lim} can be estimated as follows
\begin{align*}
&\left| \int_{B_R(x_0)} u(\tau, x)(- \Delta_{\hn})^{-s}\psi_{n}^{(x_0)}(x) \, \dg(x)
-  \int_{B_R(x_0)} u(\tau, x) \GM(x, x_0) \, \dg(x)\right| \\
& \leq \|u(\tau)\|_{L^p(\hn)}
 \|(- \Delta_{\hn})^{-s} \psi_{n}^{(x_0)} - \GM(., x_0)\|_{L^{p}(B_R(x_0))}
\xrightarrow[n \to \infty]{} 0\,.
\end{align*}
Recall that  by assumption $u_0\in L^1(\hn)\cap L^\infty(\hn)$, hence $u(t)\in L^1(\hn)\cap L^\infty(\hn)$  for all $t > 0$. This is true since WDS are $L^p$-stable, see Lemma \ref{WDS.comp+monot} and also Remark \ref{Lp-stab-WDS}.  \normalcolor
%
\medskip

Now we tackle the second term on the r.h.s. of \eqref{step2lim} and we show that
$$
\left| \int_{\hn \setminus B_R(x_0)}
u(\tau, x) (- \Delta_{\hn})^{-s} \psi_{n}^{(x_0)}(x) \, \dg(x)
-  \int_{\hn\setminus B_R(x_0)} u(\tau, x) \GM(x, x_0) \, \dg(x) \right|
\xrightarrow[n \to \infty]{} 0 \,.
$$
To this aim, we estimate $ (- \Delta_{\hn})^{-s} \psi_{n}^{(x_0)}(x)$ for $r(x_0, x) \geq R.$ Without loss of generality, we may assume $R \geq 2.$ By repeating the proof of Lemma \ref{Green-operator-estimate},  it is readily deduced that there exists a positive constant $\overline C$ (not depending on $n$) such that
\begin{equation*}
  (-\Delta_{\hn})^{-s}\,\psi_{n}^{(x_0)}(x) \leq \overline C\, \GM(x_0, x)  \quad  \text{ for all } r(x_0,x)\geq  2\,.
\end{equation*}
Therefore, invoking again Lemma \ref{Green-operator-estimate} and using the fact $u (\tau, x) \in L^1_{\Phi}(\hn)$ for each fixed $\tau$, we deduce that
$$
u(\tau, x)(- \Delta_{\hn})^{-s} \psi_{n}^{(x_0)}(x) \leq  C\,  u(\tau, x) \Phi(x) \in L^1(\hn\setminus B_R(x_0))\,
$$
for a suitable $C>0$. Hence, by Lebesgue Theorem, it follows  that
$$
\int_{\hn \setminus B_R(x_0)}
u(\tau, x) (- \Delta_{\hn})^{-s} \psi_{n}^{(x_0)}(x) \, \dg(x)\xrightarrow[n \to \infty]{}
\int_{\hn\setminus B_R(x_0)} u(\tau, x) \GH(x, x_0) \, \dg(x)\,,
$$
Writing \eqref{limit-1} with $\psi=\psi_{n}^{(x_0)}$ and recalling that the right-hand side of \eqref{limit-1} is non-negative, we have established \eqref{estimate-1}.

\medskip

{\bf Step 3 :} We make use of the time monotonicity property \eqref{mon-est} \normalcolor to prove
\begin{align}\label{eq-main-estimate}
\left( \frac{t_0}{t_1} \right)^{\frac{m}{m-1}} \, (t_1 -t_0) \int_{\hn} u^m(t_0, x) \, \psi_2(x) \,\dg(x)
&\leq \int_{t_0}^{t_1} \int_{\hn} u^m(\tau, x) \, \psi_2(x) \, \dg(x) \, {\rm d}\tau \\
& \leq  \frac{m-1}{t_0^{\frac{1}{m-1}}} \, t^{\frac{m}{m-1}} \, \int_{\hn} u^m (t, x) \, \psi_2(x) \, \dg(x)\,, \notag
\end{align}
for a.e. $0\leq t_0\leq t_1 \leq t$ and all $\psi_2 \in L_c^{\infty}(\hn)$.

Now let us start by proving the upper estimates in \eqref{eq-main-estimate}. To this aim, we consider the sequence $\psi_{1, n} \in C_c^{\infty}(0, + \infty)$ defined in Step 1. Let us choose $n$ so large that
$0 \leq t_0 - \frac{1}{n} \leq t_1 + \frac{1}{n} \leq t, $  for all nonnegative $\psi_2 \in L_c^{\infty}(\hn)$ we have
\begin{align*}
& \int_0^{\infty} \psi_{1,n}(\tau) \int_{\hn} \, u^m(\tau, x) \, \psi_2(x) \, \dg(x) \, {\rm d}\tau\\
&\leq \int_{t_0 -\frac{1}{n}}^{t_1 + \frac{1}{n}}  \, \left( \frac{t}{\tau} \right)^{\frac{m}{m-1}}
\int_{\hn} u^m (t, x) \, \psi_2(x) \, \dg(x) \, {\rm d}\tau \\
& \leq  \,\left( \int_{t_0 -\frac{1}{n}}^{t_1 + \frac{1}{n}} \,
\left( \frac{t}{\tau} \right)^{\frac{m}{m-1}} \, {\rm d}\tau \right) \,\left( \int_{\hn} u^m(t, x) \, \psi_2(x) \, \dg(x) \right) \\
& \leq (m-1) \frac{1}{(t_0 - \frac{1}{n})^{\frac{1}{m-1}}} \, t^{\frac{m}{m-1}} \,
\int_{\hn} \, u^m(t, x) \, \psi_2(x) \, \dg(x).
\end{align*}
Now, letting $n \rightarrow \infty$, we get
$$
\int_{t_0}^{t_1} \int_{\hn} \, u^m(\tau, x) \, \psi_2(x) \, \dg(x) \, {\rm d}\tau
 \leq  \frac{(m-1)}{t_0^{\frac{1}{m-1}}} \, t^{\frac{m}{m-1}} \, \int_{\hn} u^m(t, x) \, \psi_2(x) \, \dg(x),
$$
for all $t_0 \leq t_1 \leq t.$ Hence, we obtain the upper bound estimates of \eqref{eq-main-estimate}. Following a similar approach we obtain the
left-hand side of \eqref{eq-main-estimate}, see \cite[Proposition~4.2-Step 3]{BV2} for more details.

{\bf Step 4 :}  The final conclusion, i.e. \eqref{main-estimate-1}, can be obtained by an approximation procedure again. Indeed, by combining and \eqref{limit-1} and \eqref{eq-main-estimate}, we know
\begin{align}\label{final}
&\left( \frac{t_0}{t_1} \right)^{\frac{m}{m-1}} \, (t_1 -t_0) \int_{\hn} u^m(t_0, x) \, \psi_{n}^{(x_0)}(x) \, \dg(x)\\
&\int_{\hn} u(t_0, x) (- \Delta_{\hn})^{-s} (\psi_{n}^{(x_0)}(x)) \, \dg(x)
- \int_{\hn} u(t_1, x) (- \Delta_{\hn})^{-s} (\psi_{n}^{(x_0)}(x)) \, \dg(x)  \notag \\
&\leq  \frac{m-1}{t_0^{\frac{1}{m-1}}} \, t^{\frac{m}{m-1}} \, \int_{\hn} u^m (t, x) \, \psi_{n}^{(x_0)}(x) \, \dg(x) \notag
\end{align}
where the sequence $\psi_{n}^{(x_0)}$ is as defined in Step 2. So, as before, as $n \rightarrow +\infty$, we have

$$
(- \Delta_{\hn})^{-s}(\psi_{n}^{(x_0)})(x)=\frac{1}{|B_{1/n}(x_0)|} \int_{B_{1/n}(x_0)} \GM(y, x)\,\dg(x)(y)\,
\longrightarrow \, \GM(x,x_0),
$$
provided $x_0$ is the Lebesgue point of the function $x \mapsto \GM(y, x)$. Furthermore, for all $\tau \geq 0$, as $n \rightarrow +\infty$, we have
$$ \int_{\hn} u^m(\tau, x) \, \psi_{2,n}^{(x_0)}(x) \, \dg(x)
= |B_{\frac{1}{n}}(x_0)|^{-1} \int_{B_{\frac{1}{n}}(x_0)} u^m(\tau, x) \, \dg(x) \longrightarrow  u^m(\tau,x_0)\,,
$$
provided $x_0$ is the Lebesgue point of the function $x \mapsto u^m(\tau, x).$ With no loss of generality we may
choose $x_0$ belonging to the Lebesgue point set for both functions.  With this information in hand, one can let $n \rightarrow \infty$
in \eqref{final} to obtain \eqref{main-estimate-1}. This completes the proof
of Proposition~\ref{crucial}.

\end{proof}

From step 1 in the  proof of the above proposition it readily follows the monotonicity result below which is of fundamental importance to get existence of solutions in $L^{1}_{\Phi}(\hn)$.
\begin{prop}\label{monotonePhi}
Let $u \geq 0$ be a WDS to problem \eqref{NFDE} with $u_0\in L^1(\hn)\cap L^\infty(\hn)$. Then for all $\psi \in C^{\infty}_{c} (\hn)$ with $\psi\geq 0$ there holds
\begin{align}\label{potential-monotonicity}
\int_{\hn} u(t_1, x) (- \Delta_{\hn})^{-s}(\psi) (x) \, \dg(x) \leq \int_{\hn} u(t_0, x) (- \Delta_{\hn})^{-s}(\psi)(x) \, \dg(x)
\end{align}
for all $t_1>t_0\geq 0$.
As a consequence, if $\Phi \in \mathcal{W}$, the class $\mathcal{W}$ being as in \eqref{W},  the $L^1_{\Phi}$- norm is monotonically non-increasing in time, i.e.
$$\|u(t_1)\|_{L^1_{\Phi}(\hn)}\leq \|u(t_0)\|_{L^1_{\Phi}(\hn)}$$
for all $t_1>t_0\geq 0$.

\end{prop}
\begin{proof}

By definition of $\mathcal{W}$, $(\Phi \circ \tau_{-x_0})(x)=(- \Delta_{\hn})^{-s}(\psi) (\tau_{-x_0}(x))$ for some $\psi \in C^{\infty}_{c} (\hn)$ with $\psi\geq 0$ and for any $x_0\in \hn$. On the other hand, by recalling that the dependence of the heat kernel from the spatial variables is radial, we have that $\GM(x,y)$ only depends on $r(x,y)$. From this it is readily deduced that 
$$(- \Delta_{\hn})^{-s}(\psi) (\tau_{-x_0}(x)) = (- \Delta_{\hn})^{-s}(\psi_{x_0}) (x)$$
where $(\psi_{x_0}) (x):= \psi (\tau_{-x_0}(x))$. Then, since $\psi_{x_0}\in C^{\infty}_{c} (\hn)$, from \eqref{potential-monotonicity} written for $\psi=\psi_{x_0}$ we obtain
\begin{align*}
\int_{\hn} u(t_1, x) (\Phi \circ \tau_{-x_0})(x) \, \dg(x) \leq \int_{\hn} u(t_0, x) (\Phi \circ \tau_{-x_0})(x) \, \dg(x)
 \end{align*}
 and the proof follows by taking the supremum over $x_0\in \hn$.
\end{proof}

Next, we get the monotonicity for solutions in $L^{1}_{\Phi_1}(\hn)$, $\Phi_1$ being a ground state of $-\Delta_{\hn}$.

\begin{prop}\label{mon_phi1}
Let $u \geq 0$ be a WDS to problem \eqref{NFDE} with $u_0\in L^1(\hn)\cap L^\infty(\hn)$. Let $\Phi_1$ be a ground state of $-\Delta_{\hn}$. Then  the $L^1_{\Phi_1}$- norm is monotonically non-increasing in time, i.e.

$$\|u(t_1)\|_{L^1_{\Phi_1}(\hn)}\leq \|u(t_0)\|_{L^1_{\Phi_1}(\hn)}$$


for all $t_1>t_0\geq 0$.
\end{prop}

\begin{proof}
 The proof follows from Step 1 in the proof of Proposition \ref{crucial}. To this aim, let $\psi_n\in L^{\infty}_c(\hn)$  be a sequence of nonnegative functions such that  $\psi_n \rightarrow \Phi_1 \circ \tau_{-x_0} $ a.e. in $\hn$ and such that  $\psi_n \leq \Phi_1 \circ \tau_{-x_0}$, we write \eqref{limit-1} with $\psi=\psi_n$:
\begin{align*}
&\int_{\hn} u(t_0, x) (- \Delta_{\hn})^{-s} (\psi_n(x)) \, \dgh(x)
- \int_{\hn} u(t_1, x)  (- \Delta_{\hn})^{-s} (\psi_n(x)) \, \dgh(x) \\
=&  \int_{t_0}^{t_1} \int_{\hn} u^{m}(\tau, x) \psi_n(x) \, \dgh(x) {\rm d}\tau\,.
\end{align*}
Since
$$\int_{\hn} u(t_0, x) (- \Delta_{\hn})^{-s} (\psi_n(x)) \, \dgh(x) \leq \Lambda^{-s }\int_{\hn} u(t_0, x)  (\Phi_1 \circ \tau_{-x_0})(x)  \, \dgh(x) <+\infty$$
and
$$\int_{t_0}^{t_1} \int_{\hn} u^{m}(\tau, x) \psi_n(x) \, \dgh(x) {\rm d}\tau\leq \int_{t_0}^{t_1} \int_{\hn} u^{m}(\tau, x) (\Phi_1 \circ \tau_{-x_0})(x)  \, \dgh(x) {\rm d}\tau<+\infty\,,$$
we may pass to the limit and get
\begin{align*}
&\Lambda^{-s }\int_{\hn} u(t_0, x) (\Phi_1 \circ \tau_{-x_0})(x)  \, \dgh(x)
-\Lambda^{-s } \int_{\hn} u(t_1, x)  (\Phi_1 \circ \tau_{-x_0})(x)  \, \dgh(x) \\
\leq &  \int_{t_0}^{t_1} \int_{\hn} u^{m}(\tau, x) (\Phi_1 \circ \tau_{-x_0})(x)  \, \dgh(x) {\rm d}\tau \geq 0\,.
\end{align*}
This completes the proof  by taking the supremum over $x_0\in \hn$.
Notice that in the above steps we have exploited the fact that, since $\GM(x,y)$ only depends on $r(x,y)$, $\Phi_1 \circ \tau_{-x_0}$ is still a ground state. 
\end{proof}


\section{Boundedness of WDS. Proof of the Smoothing Effects}\label{smooth}

In this section we prove the three different smoothing effects for   Weak Dual Solutions, namely the results of Theorems \ref{thm.smoothing.HN-like}, \ref{smoothPhi1} and \ref{smoothPhi}.  \normalcolor

\subsection{Proof of Theorem \ref{thm.smoothing.HN-like}}

Without loss of generality, we can assume $0\le u_0\in L^1(\hn)\cap L^\infty(\hn)$, we will explain how to remove these apparent restrictions at the end of the proof. First we recall that estimate \eqref{main-estimate-1} with  $t_1 = 2t_0$, immediately implies
\begin{align}\label{thm.smoothing.HN-like.proof.1}
u^m&(t_0, x_0) 
\leq \frac{2^{\frac{m}{m-1}}}{ t_0} \int_{\hn} u(t_0, x) \, \GM(x, x_0) \, \dg(x)\\
&= \underbrace{\frac{ 2^{\frac{m}{m-1}}}{ t_0}  \int_{B_{R}(x_0)}   u(t_0, x) \GM(x, x_0)  \dg(x)}_{(I)}
+  \underbrace{\frac{  2^{\frac{m}{m-1}}}{ t_0} \int_{\hn\setminus B_{R}(x_0)}  u(t_0, x) \GM(x, x_0)  \dg(x)}_{(II)}\nonumber
\end{align}
for some $R>0$.

\noindent  {\bf  Proof of \eqref{thm.smoothing.HN-like.estimate}.} We are going to use the Green function estimates \eqref{Hyp.Green.HN0} and  \eqref{Hyp.Green.HN1} of Lemma \ref{lem.Green.HN.est} to estimate   the two  terms of inequality \eqref{thm.smoothing.HN-like.proof.1}.\\
As for the first term:
\begin{align}\label{thm.smoothing.HN-like.proof.2}
(I)&\le \frac{ 2^{\frac{m}{m-1}}}{ t_0}\|u(t_0)\|_{L^\infty(\hn)}\int_{B_{R}(x_0)} \GM(x, x_0)  \dg(x)\nonumber\\
&\le \frac{1}{m}\|u(t_0)\|_{L^\infty(\hn)}^m+\frac{ c_m}{t_0^{\frac{m}{m-1}}}\left[\int_{B_{R}(x_0)} \GH(x, x_0)  \dgh(x)\right]^{\frac{m}{m-1}}\\
&\le \frac{1}{m}\|u(t_0)\|_{L^\infty(\hn)}^m+\frac{c_m}{t_0^{\frac{m}{m-1}}} \left(\overline{\overline{\kappa}}_1\right)^{\frac{m}{m-1}}  R^{\frac{2sm}{m-1}},\nonumber
\end{align}
where we have used Young's inequality $ab\le \frac{a^m}{m}+\frac{m-1}{m}b^{\frac{m}{m-1}}$ and the estimate \eqref{Hyp.Green.HN1}. The constant $c_m=\frac{m-1}{m}2^{\frac{m^2}{(m-1)^2}}$ only depends on $m>1$.

As for the second term, from the estimate \eqref{Hyp.Green.HN0}, we derive
\begin{align}\label{thm.smoothing.HN-like.proof.3}
(II)&\le \frac{2^{\frac{m}{m-1}}}{t_0} \int_{\hn\setminus B_{R}(x_0)}  u(t_0, x) \GM(x, x_0)  \dg(x) \nonumber \\
& \le \frac{2^{\frac{m}{m-1}}}{t_0} \int_{\hn\setminus B_{R}(x_0)} u(t_0, x) \, \frac{1}{(r(x, x_0))^{N-2s}} \, \dg(x)
\leq \frac{2^{\frac{m}{m-1}} \kappa_1}{t_0}  \frac{\|u(t_0, x)\|_{L^1(\hn)}}{R^{N-2s}}.
\end{align}
Combining inequalities \eqref{thm.smoothing.HN-like.proof.1}, \eqref{thm.smoothing.HN-like.proof.2} and \eqref{thm.smoothing.HN-like.proof.3} we obtain
\begin{align*}
u^m(t_0, x_0)
&\le  \frac{1}{m}\|u(t_0)\|_{L^\infty(\hn)}^m+\frac{c_m}{t_0^{\frac{m}{m-1}}} \left(\overline{\overline{\kappa}}_1\right)^{\frac{m}{m-1}}   R^{\frac{2sm}{m-1}}+\frac{2^{\frac{m}{m-1}} \kappa_1}{t_0} \frac{\|u(t_0)\|_{L^1(\hn)}}{R^{N-2s}}\\
&\leq \frac{1}{m}\|u(t_0)\|_{L^\infty(\hn)}^m+ \kappa_4 \frac{R^{\frac{2sm}{m-1}}}{t_0^{\frac{m}{m-1}}} 
\left[ 1 + \frac{t_0^{\frac{1}{m-1}}\|u(t_0)\|_{L^1(\hn)}}{R^{\frac{1}{(m-1)\vartheta_1}}}\right]\nonumber
\end{align*}
where $\kappa_4=\kappa_4(s,m, \kappa_1)$. By taking the supremum in $x_0\in \hn$ we obtain
\begin{align}\label{thm.smoothing.HN-like.proof.5}
\|u(t_0)\|_{L^\infty(\hn)}^m \le \frac{m \, \kappa_4}{m-1}\, \frac{R^{\frac{2sm}{m-1}}}{t_0^{\frac{m}{m-1}}} \left[ 1 + \frac{t_0^{\frac{1}{m-1}}\|u(t_0)\|_{L^1(\hn)}}{R^{\frac{1}{(m-1)\vartheta_1}}}\right]\,.
\end{align}
Choosing now $R=\left(t_0^{\frac{1}{m-1}}\|u(t_0)\|_{L^1(\hn)}\right)^{(m-1)\vartheta_1}$, inequality \eqref{thm.smoothing.HN-like.proof.5} gives
\begin{align*}
\|u(t_0)\|_{L^\infty(\hn)}^m \le \frac{  2\,  m \,\kappa_4}{m-1}\,\frac{\|u(t_0)\|_{L^1(\hn)}^{2s m\vartheta_1}}{t_0^{Nm\vartheta_1}}\,,
\end{align*}
which proves inequality \eqref{thm.smoothing.HN-like.estimate}  for all $t_0>0$,  recalling that the $L^1$-norm is decreasing in time, $\|u(t_0)\|_{L^1(\hn)}\le \|u_0\|_{L^1(\hn)}$. As for the constant
\[
\ka_1^m= \frac{  2\,  m \,\kappa_4}{m-1}\,\left[1+ \cosh(1)^{\frac{m(N-1)}{m-1}}\right]\,.
\]
 
Once we have proven the estimate \eqref{thm.smoothing.HN-like.estimate} for $0\le u_0\in L^1(\hn)\cap L^\infty(\hn)$, it is easy to see that this can be extended to all $0\le u_0\in L^1(\hn)$. Indeed, given $0\le u_0\in L^1(\hn)$ consider a sequence of data $0\le u_{0,n}\in L^1(\hn)\cap L^\infty(\hn)$ converging pointwise monotonically from below (consider for instance $u_{0,n}=u_0\wedge n$) to $u_0$ and also in the strong $L^1(\hn)$ topology. Let $u_n(t) \in L^1(\hn)\cap L^\infty(\hn)$ be the corresponding solutions. Then it is clear that $u_n(t)$ converges to $u(t)$ for any fixed $t$, where $u(t)$ is the solution corresponding to $u_0$. Moreover, by lower semicontinuity of the $L^\infty$ norm and applying estimate \eqref{thm.smoothing.HN-like.estimate} to $u_n$, we conclude
\[
\|u(t)\|_{L^\infty(\hn)}\le \liminf_{n\to\infty}\|u_n(t)\|_{L^\infty(\hn)}\le \ka_0  \liminf_{n\to\infty}\frac{\|u_{0,n}\|_{L^1(\hn)}^{2s\theta_1}}{t^{N\theta_1}}= \ka_0  \frac{\|u_0\|_{L^1(\hn)}^{2s\theta_1}}{t^{N\theta_1}}.
\]

\noindent  {\bf  Proof of \eqref{thm.smoothing.HN-like.estimate.2}.} We are going to use the Green function estimates \eqref{Hyp.Green.HN0b} and  \eqref{Hyp.Green.HN2} of Lemma \ref{lem.Green.HN.est} to estimate   the two  terms of inequality \eqref{thm.smoothing.HN-like.proof.1}, when $R\ge 2$. As for the first term:
\begin{align}\label{thm.smoothing.HN-like.proof.6}
(I)&\le \frac{ 2^{\frac{m}{m-1}}}{ t_0}\|u(t_0)\|_{L^\infty(\hn)}\int_{B_{R}(x_0)} \GM(x, x_0)  \dg(x)\nonumber\\
&\le \frac{1}{m}\|u(t_0)\|_{L^\infty(\hn)}^m+\frac{ c_m}{t_0^{\frac{m}{m-1}}}\left[\int_{B_{R}(x_0)} \GH(x, x_0)  \dgh(x)\right]^{\frac{m}{m-1}}\\
&\le \frac{1}{m}\|u(t_0)\|_{L^\infty(\hn)}^m+\frac{c_m}{t_0^{\frac{m}{m-1}}}(\widetilde \kappa_1)^{\frac{m}{m-1}} R^{\frac{sm}{m-1}},\nonumber
\end{align}
where we have used Young's inequality $ab\le \frac{a^m}{m}+\frac{m-1}{m}b^{\frac{m}{m-1}}$ and the estimate \eqref{Hyp.Green.HN2}. The constant $c_m=\frac{m-1}{m}2^{\frac{m^2}{(m-1)^2}}$ only depends on $m>1$.

As for the second term:
\begin{align}\label{thm.smoothing.HN-like.proof.7}
 (II)&\le \frac{2^{\frac{m}{m-1}}}{t_0} \int_{\hn\setminus B_{R}(x_0)}  u(t_0, x) \GM(x, x_0)  \dg(x) \nonumber\\
 & \le \frac{\kappa_1 2^{\frac{m}{m-1}}}{t_0} \int_{\hn \setminus B_{R}(x_0)} \frac{ u(t_0, x) }{r(x,x_0)^{1-s}}  e^{-(N-1) r(x, x_0)}  \dg(x) \nonumber\\
&\le \frac{\kappa_1 2^{\frac{m}{m-1}}}{t_0}   \frac{\|u(t_0)\|_{L^1(\hn)}}{R^{1-s}e^{(N-1) R}},
\end{align}
where we have used  the estimate \eqref{Hyp.Green.HN0b}. Combining inequalities \eqref{thm.smoothing.HN-like.proof.1}, \eqref{thm.smoothing.HN-like.proof.6} and \eqref{thm.smoothing.HN-like.proof.7} we obtain
\begin{align*}
u^m(t_0, x_0)
&\le  \frac{1}{m}\|u(t_0)\|_{L^\infty(\hn)}^m+\frac{c_m}{t_0^{\frac{m}{m-1}}}(\widetilde \kappa_1)^{\frac{m}{m-1}} R^{\frac{sm}{m-1}}+ \frac{\kappa_1 2^{\frac{m}{m-1}}}{t_0} \frac{\|u(t_0)\|_{L^1(\hn)}}{R^{1-s}e^{(N-1) R}}\\
&=\frac{1}{m}\|u(t_0)\|_{L^\infty(\hn)}^m+ \kappa_5\,\frac{R^{\frac{sm}{m-1}}}{t_0^{\frac{m}{m-1}}}
\left[1+ \frac{t_0^{\frac{1}{m-1}}\|u(t_0)\|_{L^1(\hn)}}{R^{1-s+\frac{sm}{m-1}}e^{(N-1) R}}\right]\nonumber
\end{align*}
where $\kappa_5=\kappa_5(s,m, \kappa_1)$.
By taking the supremum in $x_0\in \hn$, recalling that we are assuming $R\ge 2$ and that $\|u(t_0)\|_{L^1(\hn)}\le \|u_0\|_{L^1(\hn)}$, we obtain
\begin{align*}
\|u(t_0)\|_{L^\infty(\hn)}^m \le \frac{m \, \kappa_5}{m-1}\,\frac{R^{\frac{sm}{m-1}}}{t_0^{\frac{m}{m-1}}}
\left[1+ \frac{t_0^{\frac{1}{m-1}}\|u_0\|_{L^1(\hn)}}{2^{1-s+\frac{sm}{m-1}}e^{(N-1) R}}\right]\,.
\end{align*}
Choosing now $R=\frac{1}{N-1}\log\left(t_0^{\frac{1}{m-1}}\|u_0\|_{L^1(\hn)}\right)$, the above inequality gives
\begin{align*}
\|u(t_0)\|_{L^\infty(\hn)}^m \le
\frac{\ka_2^m}{t_0^{\frac{m}{m-1}}}\left[ \log\left(t_0 \,\|u_0\|_{L^1(\hn)}^{m-1}\right)\right]^{\frac{sm}{m-1}}
\end{align*}
with constant
\begin{align*}
\ka_2^m:=\frac{m \, \kappa_5}{m-1}\,\left[1+ \frac{1}{2^{1+\frac{s}{m-1}}}\right]\left(\frac{1}{(N-1)(m-1)}\right)^{\frac{sm}{m-1}}.
\end{align*}
We finally remark that the assumption $R\ge 2$ and our choice of $R$, restrict the validity of the above inequality to large times, namely we have
\[
R\ge 2\qquad\mbox{if}\qquad t_0\ge e^{2(N-1)(m-1)}\|u_0\|_{L^1(\hn)}^{-(m-1)}.
\]
We have thus obtained \eqref{thm.smoothing.HN-like.estimate.2} and the proof is concluded.

\subsection{Proof of Theorem \ref{smoothPhi1}}

Let us start with the fundamental estimates :
\begin{align}\label{f2}
&u^m(t_0, x_0)
\leq \frac{2^{\frac{m}{m-1}}}{ t_0} \int_{\hn} u(t_0, x) \, \GH(x, x_0) \, \dgh(x)\\ \nonumber
 =&  \underbrace{\frac{2^{\frac{m}{m-1}}}{t_0} \int_{B_{R}(x_0)} u(t_0, x) \,  \GH(x, x_0) \,  \dgh(x)}_{ (I)}
+  \underbrace{\frac{2^{\frac{m}{m-1}}}{t_0} \int_{\hn \setminus B_{R}(x_0)} u(t_0, x) \, \GH(x, x_0) \,  \dgh(x)}_{(II)}.
\end{align}

The estimate for $(I)$ works exactly in the proof of Theorem \ref{thm.smoothing.HN-like} i.e., for all $R>0$ we have
\begin{equation}\label{w-i1}
(I) \leq \frac{1}{m}\|u(t_0)\|_{L^\infty(\hn)}^m+\frac{c_m}{t_0^{\frac{m}{m-1}}} \left(\overline{\overline{\kappa}}_1\right)^{\frac{m}{m-1}}  R^{\frac{2sm}{m-1}}
\end{equation}
where $c_m=\frac{m}{m-1}2^{\frac{m^2}{(m-1)^2}}$ and  $\overline{\overline{\kappa}}_1$  is as in Lemma \ref{lem.Green.HN.est}. Instead, if $R\geq 2$, as in \eqref{thm.smoothing.HN-like.proof.6} we get
\begin{align}\label{w-i2}
(I)\le \frac{1}{m}\|u(t_0)\|_{L^\infty(\hn)}^m+\frac{c_m}{t_0^{\frac{m}{m-1}}}(\widetilde \kappa_1)^{\frac{m}{m-1}} R^{\frac{sm}{m-1}}\,
\end{align}
with $\widetilde \kappa_1$ defined in \eqref{Hyp.Green.HN2}.

 {\bf  Proof of \eqref{eq-k3}
.}
First, combining \eqref{bound} with Corollary \ref{cor-1}, we observe that

\begin{equation}\label{gphi}
 \frac{\GH(x_0, x)}{(\Phi_1 \circ \tau_{-x_0})(x)} \leq
 \left \{ \begin{array}{ll}
 \frac{\overline C_3\,\alpha^{-1}}{(r(x_0, x))^{N-2s}} \frac{e^{\frac{(N-1)}{2} r(x_0, x)}}{ (1+r(x_0, x))} &  \quad \text{for } r(x_0, x) \leq 1\,, \\
 \frac{\overline C_4\, \alpha^{-1} \, e^{-\frac{(N-1)}{2}r(x_0, x)}}{r(x_0, x)^{1-s}(1+ r(x_0, x))} & \quad \text{for }  r(x_0, x) \geq 1\,,
\end{array}
\right.
\end{equation}
 where we have exploited the fact that $\Phi_1(x)=\Phi_1(r(x,o))$ and $r(\tau_{-x_0}(x),o)=r(x,x_0)$.  Hence, by the fact that $e^{-\frac{(N-1)r}{2}}/ r^{-N+s+1}$ is bounded for $r>1$, we infer that
\begin{equation*}
 \frac{\GH(x_0, x)}{(\Phi_1 \circ \tau_{-x_0})(x)} \leq \frac{\kappa_6}{(r(x_0, x))^{N-2s}(1+r(x_0, x))} \quad \text{for all }r(x_0, x) >0\,
\end{equation*}
for some $\kappa_6=\kappa_6(s,m, \kappa_1)$.
Using the above $(II)$ can be estimated as follows
\begin{equation*}
(II) \leq   \frac{2^{\frac{m}{m-1}}\, \kappa_6}{t_0R^{N-2s}(1+R)} \|u(t_0)\|_{L^1_{\Phi_1}(\hn)} \,.
\end{equation*}

 Now, by \eqref{w-i1} and Young's inequality we obtain
\begin{align}\label{final-i1}
\|u(t_0)\|^m_{L^{\infty}(\hn)} \le \frac{m \, \kappa_7}{m-1}\, \frac{R^{\frac{2sm}{m-1}}}{t_0^{\frac{m}{m-1}}} \left[ 1 + \frac{t_0^{\frac{1}{m-1}} 
  \|u(t_0)\|_{L^1_{\Phi_1}(\hn)}}{R^{\frac{1}{(m-1)\vartheta_1}}(1+R)}\right]
\end{align}
where $\kappa_7=\kappa_7(s,m, \kappa_1)$. If we estimate the r.h.s. of \eqref{final-i1} by using $1+R>1$, then the same argument adopted for estimating \eqref{thm.smoothing.HN-like.proof.5},  recalling that from Proposition \ref{mon_phi1} the $L^1_{\Phi_1}$-norm is non increasing in time,  yields \eqref{eq-k3}
with \[
\ka_3^m:= \frac{ 2\,  m \,\kappa_7}{m-1}\,.
\]

\medskip

 {\bf  Proof of \eqref{eq-k5}. }  Take $R\geq 2 $ in \eqref{f2}, by \eqref{gphi} we have
\begin{align}\label{w-i3}
(II)&=\frac{2^{\frac{m}{m-1}}}{t_0}\, \int_{\hn \setminus B_{R}(x_0)}  \frac{\GH(x, x_0)}{(\Phi_1 \circ \tau_{-x_0})(x)} \, u(t_0,x)\, (\Phi_1 \circ \tau_{-x_0})(x) \, \dgh(x)  \notag \\
 & \leq \frac{2^{\frac{m}{m-1}}\overline C_4 }{\alpha\,t_0} \, \frac{e^{-\frac{(N-1)}{2}R}}{R^{2-s} } \,
  \|u_0\|_{L^1_{\Phi_1}(\hn)}\,.
\end{align}
where for the last estimate we have exploited the fact that the $L^1_{\Phi_1}$-norm is decreasing in time which follows from Proposition \ref{mon_phi1}. Therefore, by \eqref{w-i2} and \eqref{w-i3}, we conclude that
\begin{align}\label{w-i4}
\|u(t_0)\|^m_{L^{\infty}(\hn)}  \leq \frac{m\, \kappa_8}{m-1} \frac{R^{\frac{sm}{m-1}}}{t_0^{\frac{m}{m-1}}} \left[ 1+
 \frac{e^{-\frac{(N-1)}{2}R}}{R^{2+\frac{s}{m-1}}}\,t_0^{\frac{1}{m-1}}\, \|u_0\|_{L^1_{\Phi_1}(\hn)} \right]\,,
\end{align}
where $\kappa_8=\kappa_8(s,m, \kappa_1)$.
Finally, we choose
$$
R = \frac{2}{(N-1)} \, \log \left( t^{\frac{1}{m-1}} \|u_0\|_{L^1_{\Phi_1}(\hn)} \right)
$$
and substituting in \eqref{w-i4} we obtain
\begin{equation*}
\|u(t_0)\|_{L^{\infty}(\hn)}^m \leq \ka_4^m \, \frac{\left[ \log( t_0\,
 \|u_0\|_{L^1_{\Phi_1}(\hn)} ^{m-1})\right]^{\frac{sm}{m-1}}}{ t_0^{\frac{m}{m-1}}}\,,
\end{equation*}
where
$$\ka_4^m:=\frac{m \, \kappa_8}{m-1}\,\left[1+ \frac{1}{2^{2+\frac{s}{m-1}}}\right]\left(\frac{2}{(N-1)(m-1)}\right)^{\frac{sm}{m-1}}.
 $$
The assumption $R\ge 2$ and our choice of $R$, restrict the validity of the above inequality to large times. More precisely, we infer that
\[
R\ge 2\qquad\mbox{if}\qquad t_0\ge e^{(N-1)(m-1)} \|u_0\|_{L^1_{\Phi_1}(\hn)}^{-(m-1)}.
\]
We have thus obtained \eqref{eq-k5} and the proof is concluded.

\medskip

\subsection{Proof of Theorem \ref{smoothPhi}}

We begin again by exploiting the fundamental estimate \eqref{f2}. The estimate for $(I)$ works exactly in the proof of Theorem \ref{thm.smoothing.HN-like} i.e., for all $R>0$ we have \eqref{w-i1} i.e.,
\begin{equation*}
(I) \leq \frac{1}{m}\|u(t_0)\|_{L^\infty(\hn)}^m+\frac{c_m}{t_0^{\frac{m}{m-1}}}  \left(\overline{\overline{\kappa}}_1\right)^{\frac{m}{m-1}}  R^{\frac{2sm}{m-1}}
\end{equation*}
where $c_m=\frac{m}{m-1}2^{\frac{m^2}{(m-1)^2}}$ and  $\overline{\overline{\kappa}}_1$  is as in Lemma \ref{lem.Green.HN.est}. Instead, if $R\geq 2$, we get \eqref{w-i2} i.e.,
\begin{align*}
(I)\le \frac{1}{m}\|u(t_0)\|_{L^\infty(\hn)}^m+\frac{c_m}{t_0^{\frac{m}{m-1}}}(\widetilde \kappa_1)^{\frac{m}{m-1}} R^{\frac{sm}{m-1}}\,
\end{align*}
with $\widetilde \kappa_1$ defined in \eqref{Hyp.Green.HN2}.
 
For $(II)$, we first write
\begin{align}\label{w-i33}
(II)=\frac{2^{\frac{m}{m-1}}}{t_0}\, \int_{\hn \setminus B_{R}(x_0)}  \frac{\GH(x, x_0)}{(\Phi \circ \tau_{-x_0})(x)} \, u(t_0,x)\, (\Phi \circ \tau_{-x_0})(x) \, \dgh(x)\,.
\end{align}
Then, recalling that $\Phi  \in \mathcal{W}$ means $\Phi=(-\Delta_{\hn})^{-s}\overline \psi $ for some positive $ \overline \psi \in C^{\infty}_c(\hn)$, by Lemma \ref{Green-operator-estimate} we infer that
\begin{equation*}
 \Phi (x) \geq 
 \left \{ \begin{array}{ll}
 \underline D  &  \quad \text{for } r(o, x) \leq \hat R\,, \\
  \underline C\,\frac{ e^{-(N-1) r(o, x)}}{(r(o, x))^{1-s}} &\quad \text{for }  r(o, x) \geq \hat R\,,
\end{array}
\right.
\end{equation*}
where $ \underline C,  \underline D$ and $\hat R$ depend on $\overline \psi$ and on the pole $o$. Finally, from Corollary \ref{cor-1} (with $\hat R$ instead of $1$) we deduce 
\begin{equation}\label{est}
\frac{\GH(x_0,x)}{(\Phi \circ \tau_{-x_0})(x)} \leq 
\left \{ \begin{array}{ll}
\frac{C_1}{r(x_0,x)^{N-2s}}&  \quad \text{for } r(x_0, x) \leq \hat R\,, \\
 C_2&\quad \text{for }  r(x_0, x) \geq \hat R\,,
\end{array}
\right.
\end{equation}
where $C_1, C_2$ and $\hat R$ depend on $\overline \psi$ and on $o$ but not on $x_0$.
\par
\medskip
{\bf  Proof of \eqref{eq-k1PHI1}.} For $R\leq 1$, \eqref{est} yields
$$\sup_{x\in \hn \setminus B_{R}(x_0)}\frac{\GH(x_0,x)}{(\Phi \circ \tau_{-x_0})(x)} \leq \frac{C_3}{R^{N-2s}}$$
for some $C_3=C_3(\overline \psi, o, N, s)>0$. Then, \eqref{eq-k1PHI1} follows by arguing as for \eqref{thm.smoothing.HN-like.proof.5} but taking $R=\left(t_0^{\frac{1}{m-1}}\|u(t_0)\|_{_{  L^1_{\Phi}(\hn) }}\right)^{(m-1)\vartheta_1}\leq \left(t_0^{\frac{1}{m-1}}\|u_0\|_{_{  L^1_{\Phi}(\hn) }}\right)^{(m-1)\vartheta_1}\leq 1$, where we use the fact that the $L^1_{\Phi}$-norm is non increasing in time by Proposition \ref{mon_phi1}.
 
\par
\medskip
{\bf  Proof of \eqref{eq-k1PHI2}.} Let now $R \geq 1$, from \eqref{est} and the fact that the $L^1_{\Phi}$-norm is non increasing in time, we have
$$
(II)\leq \frac{2^{\frac{m}{m-1}}}{t_0}C_4 \int_{\hn \setminus B_{R}(x_0)} u(t_0, x) \, (\Phi \circ \tau_{-x_0})(x)) \, {\rm d}v_{\hn}
\le \frac{2^{\frac{m}{m-1}}}{t_0}C_4 \|u_0\|_{  L^1_{\Phi}(\hn) }\,,
$$
where $C_4$ depends on $\overline \psi$ and on $o$ but not on $x_0$. 
Then, from \eqref{f2}, we obtain
\[\begin{split}
u^m(t_0,x_0)&\le \frac{1}{m}\|u(t_0)\|_{L^\infty(\hn)}^m+\frac{c_m}{t_0^{\frac{m}{m-1}}}(\widetilde \kappa_1)^{\frac{m}{m-1}} R^{\frac{sm}{m-1}}+\frac{2^{\frac{m}{m-1}}}{t_0}C_4 \|u_0\|_{  L^1_{\Phi }(\hn) }
\end{split}
\]
which, taking $R=\left(t_0 \|u_0\|_{L^1_{\Phi}(\hn)}^{m-1}\right)^{\frac{1}{sm}}\geq 1,$ gives
$$
\|u(t_0)\|_{L^{\infty}(\hn)}^m \le  \frac{\ka_6^m}{t_0}  \|u_0\|_{  L^1_{\Phi}(\hn) }
$$
for $t_0 \geq \|u_0\|_{ L^1_{\Phi}(\hn)}^{1-m}$. This completes the proof.


\section{Semigroup theory in Banach and Hilbert spaces: existence, uniqueness, contractivity and comparison. }\label{sect.semigroups} 

There are two parallel theories based on nonlinear semigroup and gradient flow techniques, respectively posed in the $L^1(\hn)$ or $H^{-s}(\hn)$ framework. The nowadays classical theories of Benilan-Crandall-Pazy-Pierre \cite{BCr,BCPbook,CP}, in the spirit of the celebrated Crandall-Liggett theorem in Banach spaces \cite{CL71}, or the Brezis-Komura  theory \cite{Brezis1,BrezisBk,Komura} of maximal monotone operators in Hilbert spaces, both apply to our problem. Therefore, it is possible to prove existence, uniqueness and comparison results of mild solutions (solutions obtained as limit of an Implicit Time Discretization -ITD- scheme, see also \cite{Ambrosio}). Such mild (or semigroup) solutions, turn out to be unique as a consequence of contractivity of the solution map $S_t$ in the respective spaces, and actually the so-called T-contractivity (contractivity with absolute value replaced by the positive part) holds and implies comparison. However, in general mild solutions have low regularity, not enough to guarantee boundedness (the energy may not be finite).   Once constructed, we have to show that mild solutions are indeed Weak Dual Solutions and hence bounded, as a consequence of the smoothing effects proven in Section \ref{smooth}.

We devote the first part of this section to briefly explain the relevant ideas of the two nonlinear semigroup approaches, omitting tedious  and standard details: we need these solutions in order to approximate the WDS that we want to construct in suitable weighted spaces, which is one of our main results. The last part of the section is devoted to the proof of the existence and uniqueness theorems for WDS, Theorem \ref{thm-existence}.

\subsection{Nonlinear Semigroup in $L^1(\hn)$. Mild VS Weak Dual Solutions. }
The basic theory in $L^1(\hn)$ developed by Benilan-Crandall-Pazy-Pierre \cite{BCr,BCPbook,CP} applies to our setting. Indeed, our aim is to solve an equation of the form $\partial_t u=-\mathcal{L}\varphi(u)$, where the operator $\mathcal{L}: {\rm dom}(\mathcal{L})\subset L^1(\hn)\to L^1(\hn)$ is a densely defined linear m-accretive  (see also Remark \ref{Lp-stab-WDS}) operator of sub-Markovian type, and the nonlinearity $\varphi(r)=r^m$ satisfies suitable ``monotonicity conditions'' that we omit here, since the nonlinearity $\varphi(r)=r^m$ with $m>1$ is the prototype example and obviously fulfills the conditions. Indeed, this theory allows to prove existence and uniqueness for a larger class of nonlinearities and operators, but we prefer to stick to the simplest nontrivial example, in order to focus on the main ideas. In the present case, the operator is $\mathcal{L}=(-\Delta_{\hn})^s$ and the nonlinearity $\varphi(r)=r^m$, and both clearly satisfy the assumptions required for Theorem 3 and 4 of \cite{CP} to hold. The same assumptions imply that the semigroup is indeed $T$-contractive, or order preserving, cf. Chapter 19.4 of \cite{BCPbook}. We summarize the above mentioned results in the following
\normalcolor

\begin{thm}[Nonlinear Semigroup in $L^1(\hn)$, \cite{BCr,BCPbook,CP}]\label{BCPP-Thm}
Let $u_0,v_0\in L^1(\hn)$. Then there exist unique mild solutions $u,v\in C^0([0,\infty)\,:\, L^1(\hn))$ to Problem  \ref{NFDE}, such that
\begin{equation*}
\int_{\hn} \big( u(t,x)-v(t,x)\big)_+ \dg(x)\le \int_{\hn} \big( u(0,x)-v(0,x)\big)_+ \dg(x)\qquad\mbox{for all }t\ge 0\,,
\end{equation*}
and the same holds replacing $(\cdot)_+$ with $(\cdot)_-$. As a consequence,
\begin{equation*}
\|u(t)-v(t)\|_{L^1(\hn)}\le \|u_0-v_0\|_{L^1(\hn)}\qquad\mbox{for all }t\ge 0.
\end{equation*}
Moreover, nonnegative mild solutions enjoy the following time monotonicity property:
\begin{equation*}
\mbox{the map}\qquad t \mapsto t^{\frac{1}{m-1}} u(t,x) \qquad\mbox{is nondecreasing in $t > 0$ for a.e. $x \in \hn.$}
\end{equation*}
\end{thm}
As we have already explained above, now that we have built unique mild solutions in $L^1(\hn)$, we need to show that these are also WDS, hence bounded, as a consequence of Theorem~\ref{thm.smoothing.HN-like}. We devote the rest of this subsection to the proof of this fact, borrowing ideas from  Section 7 of \cite{BV1}.

\begin{thm}\label{thm-weak-mild}
Let $u$ be the mild solution to \eqref{NFDE} corresponding to the non-negative initial datum $u_0 \in L^1(\hn)$.
Then, $u$ is a WDS in the sense of Definition~\ref{defi_WDS} and it is bounded, in particular it satisfies the smoothing effects of Theorem \ref{thm.smoothing.HN-like}.
\end{thm}
To prove the above theorem we need the following auxiliary result:
\begin{prop}[$L^p$-stability of Mild  Solutions]
Let $0\le u_0 \in L^1(\hn) \cap L^p(\hn)$. Then, the unique mild   solution to \eqref{NFDE} satisfies
\begin{equation}\label{eq-wm-1}
\|u(t)\|_{L^p(\hn)} \leq \|u_0\|_{L^p(\hn)} \qquad\mbox{for all $t\ge 0$ and all $1 \le p \leq \infty.$}
\end{equation}
\end{prop}

\begin{rem}\label{Lp-stab-WDS}\textsl{$L^p$ stability of very weak solutions. }\rm
Let us remark that the $L^p$ stability is easily enjoyed by WDS. Indeed, first we have to notice that WDS are indeed very weak (or distributional) solutions, namely they satisfy the following identity for every test function $\varphi\in C_c^\infty((0,T)\times \hn)$:
\begin{equation}\label{def.vwsol}
\int_{0}^{T} \int_{\hn} u(t,x) \, \partial_t \varphi(t,x) \,
\dg \, {\rm d}t - \int_{0}^{T} \int_{\hn} u^m(t,x) (- \Delta_{\hn})^{s}\varphi(t,x) \, \dg\, {\rm d}t = 0
\end{equation}
The idea to show that WDS are very weak solutions is to put $\psi=(- \Delta_{\hn})^{s}\varphi$ in the integral identity of Definition \ref{defi_WDS} of WDS: however, to make it rigorous some approximations are needed. Note that even if $0\le \varphi\in C^\infty_c$, we always have that $(- \Delta_{\hn})^{s}\varphi\asymp \Phi\asymp e^{(N-1)r}r^{s-1}$, see Lemma \ref{Green-operator-estimate}. This is a peculiarity of the nonlocal case, in contrast with the local case $s=1$, for which ${\rm supp}(\mathcal{L}u)\subseteq {\rm supp}(u)$. For this reason, we need to require further integrability conditions also for very weak solutions, for instance $u \in L^1_{loc}((0, \infty) : L^{1}_{loc} (\hn))$ and that $u^m\in L^1_{loc}\left((0, T) : L^{1}_{\Phi}(\hn) \right)$. It is not clear wether or not WDS and very weak solutions are indeed the same class.

For very weak solutions, the $L^p$ stability is just a consequence of the m-accretivity of the operator $\mathcal{L}=(-\Delta_{\hn})^s$, which reads: if $\beta$ is a maximal monotone graph in $\RR\times\RR$, with $0\in \beta(0)$, $u\in {\rm dom}(\mathcal{L})$, $\mathcal{L}u\in L^p(\hn)$, with $1\le p\le \infty$, and $v\in L^{\frac{p}{p-1}}(\hn)$, $v(x)\in \beta(u(x))$ a.e., then
\[
\int_{\hn} v(x)\mathcal{L}u(x) \dg(x) \ge 0.
\]
In particular, it is not difficult to see that if $0\le u\in {\rm dom}(\mathcal{L})\cap L^1(\hn)\cap L^\infty(\hn)$, we always have
\begin{equation}\label{m-Accretivity.2}
\int_{\hn} u^q(x)\mathcal{L}u^m(x) \dg(x) \ge 0\qquad\mbox{for all $q,m>0$.}
\end{equation}
Let now $u$ be a very weak solution of $u_t=-\mathcal{L}u^m$, corresponding to $0\le u_0 \in C^\infty_c(\hn)\subset {\rm dom}(\mathcal{L})\cap L^1(\hn)\cap L^\infty(\hn)$.
Inequality \eqref{eq-wm-1} then follows formally by integrating the following differential inequality for any $p\in [1,\infty]$,
\[
\frac{d}{dt}\int_{\hn} u^p(t,x)\dg(x)=-p \int_{\hn} u^{p-1}(t,x)\mathcal{L}u^m(t,x)\dg(x) \le 0\,,
\]
where in the last step we have used \eqref{m-Accretivity.2}.
The above proof can be made rigorous through a long but standard approximation argument:  the above differential inequality (integrated on $[0,T]$),  is nothing but the definition of very weak solution \eqref{def.vwsol} with the choice of test function $\psi(t,x)=u^{p-1}$\,. Of course, that choice is not admissible, but it can be approximated by a sequence of admissible test functions $\psi_n\in C_c^\infty((0,\infty)\times \hn)$  through a careful but standard limiting process.

We now sketch the proof -which follows the ideas of Proposition 7.1 of \cite{BV1}- of the $L^p$ stability for mild solutions, which only relies on the definition of mild solutions (without passing through very weak solutions, hence avoiding tedious and long approximations).
\end{rem}\normalcolor
\begin{proof} The case $p=1$ follows directly by Theorem \ref{BCPP-Thm}. Fix $1 < p < \infty$ and for $n\geq 1$ set $u_{0, n}(x) := T_{n} u_0(x)$ where $T_n : \mathbb{R} \rightarrow \mathbb{R}$ is defined as follows
\begin{equation}\label{tn}
T_n(x) =
 \left \{ \begin{array}{ll}
   x & \mbox{if} \ |x| \leq n\,,  \\
 \frac{x}{|x|} \, n  &  \mbox{if } \ |x|  > n.
\end{array}
\right.
\end{equation}
Clearly, $u_{0, n} \in L^p(\hn) \subset L^1(\hn)\cap L^{\infty} (\hn)$ and $u_{0, n} \rightarrow u_0$ a.e. in $\hn$, hence by dominated convergence
$u_{0, n} \rightarrow u_0$ strongly in $L^{p}(\hn)$.

We claim that it is sufficient to show that for any fixed $t>0$, the semigroup solution $u_n(t)$ corresponding to the initial datum $u_{0, n}$ satisfies
\begin{equation}\label{eq-claim-1}
\|u_n(t)\|_{L^p(\hn)} \leq \|u_{0, n}\|_{L^p(\hn)} \leq \|u_0\|_{L^p(\hn)}\,
\end{equation}
to conclude the proof of the Proposition, since the above inequality easily implies \eqref{eq-wm-1}. Indeed,
since $u_n(t)$ is bounded in $L^p(\hn)$, then, up to subsequences, $u_n  \rightharpoonup u$ in $L^p(\hn)$ and,
using lower semicontinuity of weak convergence, we obtain
$$
\|u(t)\|_{L^p(\hn)} \leq \liminf_{n} \|u_n(t)\|_{L^p(\hn)} \leq \|u_0\|_{L^{p}(\hn)}\,.
$$
Hence, \eqref{eq-wm-1} holds true for all $p < \infty.$ For $p = \infty$ we just let $p \rightarrow \infty$ in the above inequality.

\par

It just remains to prove \eqref{eq-claim-1}, and without loss of generality we can assume $u_0 \in L^{1}(\hn) \cap L^{\infty}(\hn)$. Mild solutions are obtained via an ITD scheme: consider the following partition of $[0, T]$ (with $t_0 = 0$ and $t_n = T$)
$$
t_k = \frac{k}{n} T, \quad \mbox{for any} \ 0 \leq k \leq n,
$$
with $h = t_{k+1} - t_{k} = \frac{T}{n}.$ For any $t \in (0, T),$ the unique semigroup solution $u(t, .)$ is obtained
as the strong $L^1(\hn)$-limit of solutions $u_{k + 1}(.) = u(t_{k+1},)$ to the following elliptic equation:
\begin{equation}\label{eq-approximate}
h (-\Delta_{\hn})^s (u_{k+1})^m  \, + \, u_{k+1} = u_{k}  \quad \mbox{in} \ \hn
\end{equation}
whose solvability is guaranteed by the running assumption on the operator and nonlinearity, cf. Theorem \ref{thm-weak-mild}, in particular
\begin{align*}
 \int_{\hn} (u_{k+1}^p - u_k u_{k+1}^{p-1} ) \, \dg
  & = -h \int_{\hn} u^m_{k+1}\, (-\Delta_{\hn})^s(u_{k+1})^{p-1}\, \dg \geq 0.
\end{align*}
(This is equivalent to the sub-Markovianity of $\mathcal{L}$). Therefore, we obtain
\begin{align*}
\int_{\hn} u_{k+1}^p \, \dg & \leq \int_{\hn} u_k \, u_{k+1}^{p-1} \, \dg \\
&  \leq \|u_k\|_{L^p(\hn)} \left( \int_{\hn} u_{k+1}^p \, \dg \right)^{\frac{p-1}{p}}.
\end{align*}
A simple inductive argument shows that
$$
\|u_{k+1}\|_{L^p(\hn)} \leq \|u_k\|_{L^p(\hn)} \leq ..... \leq \|u_0\|_{L^p(\hn)}.
$$
We conclude the proof by letting $k \rightarrow \infty$ to obtain
$$
\|u(t)\|_{L^p(\hn)} \leq \liminf_{k \rightarrow \infty} \|u_{k+1}(t)\|_{L^p(\hn)} \leq \|u_0\|_{L^p(\hn)}.
$$
\end{proof}

\noindent\textbf{Proof of Theorem \ref{thm-weak-mild}. }We sketch the proof which follows the ideas of Proposition 7.2 of \cite{BV1} which deals with FPME on bounded Euclidean domains. We will just emphasize the main points for convenience of the reader. \normalcolor We first pass to a weak dual formulation of the approximate problem elliptic problem \eqref{eq-approximate}, i.e. we apply $(-\Delta_{\hn})^{-s}$ to both members of \eqref{eq-approximate} to get:
$$
(-\Delta_{\hn})^{-s} u_{k+1} - (-\Delta_{\hn})^{-s} u_k = -h (u_{k+1})^m \quad \text{in } \hn\,.
$$
By multiplying the above equation by $\psi \in C^1_c(0, T; L_c^{\infty}(\hn)),$ we obtain (upon using the notation $\psi_k=\psi(\cdot,t_k )$)
$$
\sum_{k =0}^{n-1} \, \int_{\hn} [ (-\Delta_{\hn})^{-s} u_{k+1} - (- \Delta_{\hn})^{-s} u_k ] \, \psi_{k} \, \dg
= - h \, \sum_{k=0}^{n-1} \, \int_{\hn} (u_{k+1})^m \, \psi_{k} \, \dg.
$$
Since $\psi$ is compactly supported in time, then for sufficiently large $n$, we have
$$
\psi_1 = \psi(\cdot, t_1) = \psi(\cdot, \frac{T}{n}) = 0 \quad \mbox{and} \  \psi_{n-1} \rightarrow   \psi(\cdot, T) = 0.
$$
By this  we infer that
\begin{align*}
\int_{\hn} \, [\psi_{n-1} \, (-\Delta_{\hn})^{-s} u_n \, & - \, \psi_0 (-\Delta_{\hn})^{-s} u_0] \, \dg \\
- \underbrace{\sum_{k =1}^{n-1} \, \int_{\hn} (\psi_k - \psi_{k-1}) (-\Delta_{\hn})^{-s} u_k \, \dg}_{= I_1}
& = - \, \underbrace{h \, \sum_{k =0}^{n-1} \int_{\hn} (u_{k+1})^m \, \psi_k(x) \, \dg}_{= I_2}.
\end{align*}
Let us carefully look at $I_1$ :
\begin{align*}
I_1 & =  \sum_{k =1}^{n-1}\,  h \, \int_{\hn} \frac{(\psi_k - \psi_{k-1})}{h} (-\Delta_{\hn})^{-s} u_k \, \dg \\
& = \sum_{k =1}^{n-1}\,  h \, \int_{\hn} (\psi_k)_t(x, \hat{t}_k) \, (-\Delta_{\hn})^{-s}( u_k - u)\, \dg
+  \sum_{k =1}^{n-1}\,  h \, \int_{\hn} (\psi_k)_t(x, \hat{t}_k) \, (-\Delta_{\hn})^{-s} \, u\, \dg.
\end{align*}
Now consider the above two integrals separately. Since $\psi$ is compactly supported, we obtain
\begin{align*}
&\sum_{k =1}^{n-1}\,  h \, \int_{\hn} (\psi_k)_t(\cdot, \hat{t}_k) \, (-\Delta_{\hn})^{-s}( u_k - u)\, \dg \\
 =& \sum_{k =1}^{n-1}\,  h \, \int_{\hn} (-\Delta_{\hn})^{-s} ( (\psi_k)_t(\cdot, \hat{t}_k)) \, ( u_k - u)\, \dg \\
& \leq   \|u_k - u\|_{L^1(\hn)}  \,
 \sum_{k =1}^{n-1} \, h  \left\|(-\Delta_{\hn})^{-s} ( (\psi_k)_t(\cdot, \hat{t}_k))
  \right\|_{L^{\infty}(K)} \xrightarrow[n \to \infty]{}0\,,
\end{align*}
where we have exploited the fact that, $\|u_k - u\|_{L^1(\hn)} \xrightarrow[n \to \infty]{}0$ and that,
$$
\sum_{k =1}^{n-1} \, h  \left\|(-\Delta_{\hn})^{-s} ( (\psi_k)_t(\cdot, \hat{t}_k))
  \right\|_{L^{\infty}(K)} \xrightarrow[n \to \infty]{}  \int_{0}^{T}\|(-\Delta_{\hn})^{-s} \psi_t(\cdot,t) \|_{L^{\infty}(\hn)} dt  < \infty\,,
$$
the latter integral being bounded by arguing  as in Lemma \ref{Green-operator-estimate}.

On the other hand, we recognize that the second integral is nothing but a Riemann sum, so that
$$
 \sum_{k =1}^{n-1}\,  h \, \int_{\hn} (\psi_k)_t(x, \hat{t}_k) \, (-\Delta_{\hn})^{-s} \, u\, \dg
\longrightarrow \int_0^T \int_{\hn} \, (\psi_t) (-\Delta_{\hn})^{-s} (u) \, \dg \, {\rm d}t\,,
$$
as $n \rightarrow \infty$. Hence,
\begin{equation}\label{Riemann-1}
I_1\xrightarrow[n \to \infty]{}0 \int_0^T \int_{\hn} \, (\psi_t) (-\Delta_{\hn})^{-s} (u) \, \dg \, {\rm d}t\,.
\end{equation}
A similar computation yields
\begin{equation}\label{Riemann-2}
I_2 \xrightarrow[n \to \infty]{}0 \int_0^T \int_{\hn} \, u^m \, \psi \, \dg \, {\rm d}t\,.
\end{equation}
Moreover, letting $n \rightarrow \infty,$  we obtain
\begin{equation}\label{Riemann-3}
\int_{\hn} \, [\psi_{n-1} \, (-\Delta_{\hn})^{-s} u_n \,  - \, \psi_0 (-\Delta_{\hn})^{-s} u_0] \, \dg\,
\xrightarrow[n \to \infty]{}0 \,.
\end{equation}
Finally, as a consequence of \eqref{Riemann-1}, \eqref{Riemann-2} and \eqref{Riemann-3} we obtain
$$
\int_0^T \int_{\hn} \, (\psi_t) (-\Delta_{\hn})^{-s} (u) \, \dg \, {\rm d}t
- \int_0^T \int_{\hn} \, u^m \, \psi \, \dg \, {\rm d}t = 0\,,
$$
for all $\psi \in C^1_c(0, T; L_c^{\infty}(\hn)).$ The above expression shows that $u$ is a WDS, recalling that $L^1(\hn) \subset L^1_{\Phi}(\hn)$ and $L^1(\hn) \subset L^1_{\Phi_1}(\hn)$. Furthermore, the boundedness of solutions follows by Theorem \ref{thm.smoothing.HN-like}, which holds for WDS with $0\le u_0\in L^1(\hn)$.\qed

\subsection{Nonlinear Semigroup in $H^{-s}(\hn)$ and Gradient Flows.}\label{gradient}

The celebrated theory of maximal monotone operators in Hilbert spaces of Brezis and Komura, \cite{Brezis1,BrezisBk,Komura}, applies in this framework, see for instance Brezis' paper \cite{Brezis1}. The link between such theory and the theory of Gradient Flows has been very well explained by Ambrosio et al. in \cite{Ambrosio}, see also \cite{AGSbook}. We will briefly explain how the previously mentioned theory applies to the present case.

We choose as Hilbert space $H^{-s}(\hn)$, the dual of $H^s(\hn)$, with the Hilbertian norm given by
\begin{equation*}
\|f\|_{H^{-s}(\hn)}^2=\int_{\hn} f(-\Delta_{\hn})^{-s} f \dg = \int_{\hn} \left|(-\Delta_{\hn})^{-s/2} f\right|^2 \dg
\end{equation*}
and the corresponding scalar product. Define the (convex, nonlinear) energy functional
\begin{equation*}
E_m[f]=\frac{1}{m+1}\int_{\hn} |f|^{m+1} \dg
\end{equation*}
whenever finite, $+\infty$ elsewhere. The equation $ u_t = -(-\Delta)^s u^m$ can be interpreted as the subdifferential inclusion $u_t \in \partial[E_m](u)$, where $\partial[E_m]$ is the sub-differential in $H^{-s}(\hn)$ of the convex energy functional $E_m$.

We summarize in the following theorem the results of the  Brezis-Komura theory.

\begin{thm}[Brezis-Komura, \cite{Brezis1,BrezisBk,Komura}]\label{BK-Thm}
For every $u_0\in \overline{{\rm dom }E_m} = H^{-s}(\hn)$, there exists a unique gradient flow solution starting from $u_0$, that we denote by $u(t)=S_tu_0$. This defines a strongly continuous semigroup $S_t: H^{-s}(\hn)\to H^{-s}(\hn)$ for $t>0$, i.e. $S_{t+\tau}=S_t S_\tau$ for any $t,\tau>0$ with the contraction property
\[
\|u(t)-v(t)\|_{H^{-s}(\hn)}\le \|u_0-v_0\|_{H^{-s}(\hn)}\qquad\mbox{for all }u_0,v_0\in H^{-s}(\hn).
\]
Moreover, the T-contraction property holds true,
\begin{equation}\label{T-contraction.2}
\|(u(t)-v(t))_+\|_{H^{-s}(\hn)}\le \|(u_0-v_0)_+\|_{H^{-s}(\hn)}\qquad\mbox{for all }u_0,v_0\in H^{-s}(\hn).
\end{equation}
Hence standard comparison holds in $H^{-s}(\hn)$. Moreover,  for a.e. $t>0$ we have
\begin{equation*}
\frac{1}{2}\frac{d}{dt}\|u(t)\|_{H^{-s}(\hn)}^2= -\frac{1}{m+1}\int_{\hn} |u(t)|^{m+1} \dg = -E_m[u(t)]\le 0
\end{equation*}
and
\begin{equation*}
\frac{d}{dt}E_m[u(t)]= - \int_{\hn}u^m (-\Delta_{\hn})^s u^m \dg \le 0\,.
\end{equation*}
Finally, the equation $u_t=(-\Delta_{\hn})^s u^m$ holds true for almost every $t>0$ as equality among functions in $H^{-s}(\hn)$.   We call this solutions $H^{-s}$-strong solutions.
\end{thm}
The concept of gradient flow solution corresponds in this case to $H^{-s}$-strong solutions, which are a priori, a smaller class than mild solutions, hence more regular. 
More details can be found in \cite{Ambrosio,BrezisBk,Komura}. The proof of inequality \eqref{T-contraction.2} is due to Brezis and it has been sketched in Lemma \ref{comp.H-s}.\normalcolor

The proof of this theorem, considered nowadays the nonlinear analogous of the Hille-Yosida Theorem, follows by a careful analysis of the convergence of the ITD, or the Implicit Euler Scheme. An excellent exposition is given in the Lecture Notes of Ambrosio et al. \cite{Ambrosio}, where a characterization in terms of EVI, Evolution Variational Inequalities, is given. Recall that an absolutely continuous curve in $u(t)\in H^{-s}(\hn)$ is called an EVI solution if  for any $y\in {\rm dom }E_m\subset H^{-s}(\hn)$ we have
\begin{equation}\label{EVI}
\frac{1}{2}\frac{d}{dt}\|u(t)-y\|_{H^{-s}(\hn)}^2+ E_m[u(t)]\le E_m[y]\qquad\mbox{for a.e. }t>0.
\end{equation}
In Section 11 of  \cite{Ambrosio} it is carefully explained that $u$ is an EVI solution if and only if it is a gradient flow solution as in Theorem \ref{BK-Thm}. Hence we just have to check inequality \eqref{EVI}, namely
\begin{equation*}
\begin{split}
\frac{1}{2}\frac{d}{dt}&\|u(t)-y\|_{H^{-s}(\hn)}^2 =\int_{\hn} u_t(t)(-\Delta_{\hn})^{-s} (u(t)-y) \dg
 =\int_{\hn} (u-y) (-\Delta_{\hn})^{-s}u_t \dg\\
 &=- \int_{\hn} (u-y) u^m \dg
  =- \int_{\hn} u^{m+1} \dg + \int_{\hn}  y u^m \dg  \\
  &\le -\left(1- \frac{m}{m+1}\right) \int_{\hn} u^{m+1} \dg +\frac{1}{m+1}\int_{\hn} y^{m+1} \dg
  =-E_m[u(t)]+E_m[y]\,,
\end{split}
\end{equation*}
where we have used Fubini's theorem, the dual equation $(-\Delta_{\hn})^{-s}u_t=-u^m$, and Young's inequality $ab\le \frac{m}{m+1}a^{\frac{m+1}{m}}+\frac{1}{m+1}b^{m+1}$.
\begin{rem}\label{rem.strong.H-s}\textbf{Strong $H^{-s}$ solutions are WDS. }When $u_0$ is sufficiently integrable, for instance $u_0\in L^1(\hn)\cap L^\infty(\hn)$ (weaker assumptions are possible), then the gradient flow solutions (or strong $H^{-s}$ solutions) constructed in Theorem \ref{BK-Thm}, first introduced by Brezis and Komura \cite{Brezis1,BrezisBk,Komura}, turn our to be also WDS. Let us recall the definition of strong $H^{-s}$ solution: let $u\in C^0([0,T]:H^{-s}(\hn))$, $u^m\in L^1([0,T]: H^{s}(\hn))$, be such that
\begin{equation*}
\int_0^T\int_{\hn} u\,\partial_t \psi\, \dg dt=\int_0^T\int_{\hn}u^m\, (-\Delta_{\hn})^{s}\psi\dg dt\qquad \forall \psi\in C^1_c([0,T]:H^{s}(\hn))
\end{equation*}
or, equivalently, since $(-\Delta_{\hn})^{s}:H^{s}(\hn)\to H^{-s}(\hn)$ is an isomorfism, let $(-\Delta_{\hn})^s\psi=\eta$,
\begin{equation*}
\int_0^T\int_{\hn} (-\Delta_{\hn})^{-s}u\,\partial_t\eta\,\dg dt
=\int_0^T\int_{\hn}u^m\,\eta\,\dg dt\qquad\forall\eta\in C^1_c([0,T]:H^{-s}(\hn)).
\end{equation*}
The above solutions turn out to be equivalent to Weak Dual Solutions in the sense of Definition \ref{defi_WDS}, when the initial datum is  sufficiently integrable. This can be seen using the Poincar\'e inequality $\lambda_1\|f\|_{H^{-s}(\hn)}\le \|f\|_{L^2(\hn)}$, together with the $L^p$ stability of very weak solutions, see Lemma \ref{eq-wm-12} and Remark \ref{Lp-stab-WDS}. As a consequence, we have that $H^{-s}$-strong solutions are also WDS: this is clear since the class of test functions of WDS is strictly contained in the above class of admissible test functions, more precisely $C^1_c([0,T]:H^{-s}(\hn))\subset C^1_c([0,T]:L^\infty_c(\hn))$.
%
%

Notice that the concept of WDS is somehow more general than the one of $H^{-s}$ strong solutions, indeed nonnegative $H^{-s}$ functions are in $L^1_{\Phi}$, namely for any $0\le f \in H^{-s}(\hn) $ we have
\begin{align}\label{L1-H-s.norms}
\|f\|_{L^1_{\Phi}(\hn)}&=\int_{\hn}f(-\Delta_{\hn})^{-s}\psi\dg
=\int_{\hn}(-\Delta_{\hn})^{-s/2}f(-\Delta_{\hn})^{-s/2}\psi\dg\\
&\le \left(\int_{\hn}\left|(-\Delta_{\hn})^{-s/2} \psi\right|^2 \dg\right)^{\frac{1}{2}}
\left(\int_{\hn}\left|(-\Delta_{\hn})^{-s/2} f\right|^2\dg\right)^{\frac{1}{2}}
=c_\psi \|f\|_{H^{-s}(\hn)}\,.\nonumber
\end{align}
Hence $H^{-s}_+\subseteq L^1_{\Phi,+}(\hn)$.  Note that $c_\psi =\int_{\hn}\psi (-\Delta_{\hn})^{-s}\psi\dg
= \int_{\hn}\psi \Phi\dg = \|\psi\|_{L^1_{\Phi}}<+\infty$ since $0\le \psi\in C_c^\infty(\hn)$ and $\Phi=(-\Delta_{\hn})^{-s}\psi$.

We would like to stress that as a consequence of our $L^1_{\Phi}-L^\infty$ smoothing effects, Theorem \ref{smoothPhi}, we have shown that $H^{-s}$-strong solutions are bounded, see Remark \ref{rem.smooth.L1.H-s}.
\end{rem}

\subsection{Proof of Theorem \ref{thm-existence}} The proof of existence of (minimal) weak dual solutions will be obtained by approximations in terms of mild solutions, more precisely the minimal weak dual solution will be the limit of a pointwise monotone sequence of (unique) mild solutions: we have decided to take, as approximating sequences, $L^1$-mild solutions, whose existence and uniqueness is guaranteed by Theorem \ref{BCPP-Thm}. However, a similar proof would hold by replacing $L^1$-mild solutions with $H^{-s}$-gradient-flow solutions,  whose existence and uniqueness is guaranteed by Theorem \ref{BK-Thm}.

\noindent\textbf{Approximation: }Let $u_0\ge 0$ be the initial datum (the same construction will work in $L^1(\hn)$, $L^1_{\Phi_1}(\hn)$ or $L^1_{\Phi}(\hn)$). Define, for $n\geq 1$, $u_{0,n}(x):=\chi_{B_n(x_0)}(x)\,T_n u_0(x)$, where $T_n$ is defined as in \eqref{tn}. Hence, $u_{0,n}\in L^1(\hn) \cap L^{\infty}(\hn)$ and $u_{0,n}(x)\leq u_{0,n+1}(x)\leq u_{0}(x)\in  L^1(\hn)$ for a.e. $x \in \hn$. By comparison, we know that the unique $L^1$-mild solutions satisfy: $u_n(t,x)\le u_{n+1}(t,x)$ for a.e. $x\in\hn$ and a.e. $t>0$, hence the limit
\begin{equation}\label{def.WDS.approx}
u(t,x):=\liminf_{n\to\infty}u_{n+1}(t,x)
\end{equation}
exists for a.e. $x\in\hn$ and a.e. $t>0$ as limit of monotone sequences of real numbers. This is our candidate to be the minimal weak dual solution. We will analyze different setups below. As a consequence of the Monotone Convergence Theorem, we have that $u_n(t)\to u(t)$ as $n\to \infty$ in the strong $L^1(\hn)$, $L^1_{\Phi_1}(\hn)$ or $L^1_{\Phi}(\hn)$) topology, depending on where $u_0$ belongs.

\noindent\textbf{Proof of Theorem \ref{thm-existence}. The case $u_0\in L^1(\hn)$. }First we show that the WDS defined in \eqref{def.WDS.approx} is bounded, more precisely, by semicontinuity of the $L^\infty$-norm, we have that for all $t>0$
\[
\|u(t)\|_{L^\infty(\hn)}\leq \liminf_{n\rightarrow +\infty}\|u_n(t)\|_{L^\infty(\hn)}
\leq \ka_1\liminf_{n\rightarrow +\infty}\frac{\|u_{0,n}\|^{2s \vartheta_1}_{L^1(\hn)}}{t^{N \vartheta_1}}
= \ka_1 \frac{\|u_{0}\|^{2s \vartheta_1}_{L^1(\hn)}}{t^{N \vartheta_1}}\,,
\]
and we have used the $L^1-L^\infty$ smoothing estimates \eqref{thm.smoothing.HN-like.estimate}, which hold true for the mild solutions of the approximating sequence $u_n$, as shown in Theorem \ref{thm-weak-mild}.

As a consequence, for all $t>0$, $u(t,x) \in L^1(\hn) \cap L^{\infty}(\hn)\subset L^p(\hn)$ for all $p\geq 1$. Now, for $1<p<+\infty$, the elementary inequality $a^p-b^p\leq pa^{p-1}(a-b)$ gives
\begin{equation}\label{L1Lp}
0\leq \int_{\hn} ( u^p(t,x)-u_n^p(t,x))  \, \dg(x) \leq p \|u(t, x)\|_{L^\infty(\hn)}^{p-1} \int_{\hn} ( u(t,x)-u_n(t,x))  \, \dg(x)\,.
\end{equation}
Hence, $u_{n}(t) \rightarrow u(t)$ strongly in $L^p$, for all $p\in (1,\infty)$ and $t>0$, as a consequence of the $L^1$-strong convergence.

We are finally ready to verify that $u$ is a weak dual solution in the sense of Definition~\ref{defi_WDS}:
\begin{itemize}[leftmargin=*]
\item $u \in C([0, T) :  L^{1} (\hn))$. This follows by noting that  for all $0\leq \tau< t$ or $0\leq t<\tau$  and $n\geq n_0>0$, one has
\[
\begin{split}
& \int_{\hn} | u(t,x)-u(\tau,x)| \, \dg(x) 
 \leq \int_{\hn} |u(t,x)-u_{n}(t,x)|  \, \dg(x) \\
&   +2\int_{\hn} | u_{0,n}(x)-u_{0,n_0}(x)|  \, \dg(x) 
+\int_{\hn} |u_{n_0}(t,x)-u_{n_0}(\tau,x)|  \, \dg(x) \\
& +\int_{\hn} |u_{n}(\tau,x)-u(\tau,x)|  \, \dg(x)\,,
\end{split}
\]
 where we have also exploited the stability inequality given in Theorem \ref{BCPP-Thm}.
 Then, recalling that $u_{n_0} \in C([0, \infty) : L^1(\hn))$, the continuity for $t>0$ readily follows from above  by taking $n_0$ sufficiently large;  similarly, for $t=0$ we   also obtain that $u(0,x)=u_0(x)$ a.e. in $\hn$, by
\[
\begin{split}
\int_{\hn} |u(t,x)-u_0(x)|  &\, \dg(x)
\leq \int_{\hn} | u(t,x)-u_{n}(t,x)|  \, \dg(x) \\
&+\int_{\hn} |u_{n}(t,x)-u_{n,0}(x) | \, \dg(x)+\int_{\hn} |u_{n,0}(x)-u_0(x)|  \, \dg(x)\,.
\end{split}
\]

\item   $u^m \in L^1\left( (0, T) : L^{1}(\hn) \right)$ for any fixed $T>0$. To this aim, we note that, from inequality \eqref{thm.smoothing.HN-like.estimate} of Theorem \ref{thm.smoothing.HN-like}, we have
\begin{align*}
&\int_0^T \int_{\hn} u^m(t,x)  \, \dg(x)\,dt \leq \int_0^T\|u(t, x)\|_{L^\infty(\hn)}^{m-1}  \int_{\hn} u(t,x)  \, \dg(x)\,dt\\
& \leq \ka_1^{m-1} \|u_{0}(x)\|^{2s \vartheta_1(m-1)+1}_{L^1(\hn)} \int_{0}^{T}  \frac{1}{t^{N \vartheta_1(m-1)}} \,dt <+\infty\,,
\end{align*}
where the latter integral is bounded since $N \vartheta_1(m-1)<1$.

\item $u$ satisfies the following integral identity
\begin{equation*}
\int_{0}^{T} \int_{\hn} (- \Delta_{\hn})^{-s}(u) \, \frac{\partial \psi}{\partial t} \,
\dg \, {\rm d}t - \int_{0}^{T} \int_{\hn} u^m \psi \, \dg\, {\rm d}t = 0
\end{equation*}
for every test function $\psi \in C^1_c(0,T; \mbox{L}_c^{\infty}(\hn)).$ This follows by recalling that we have proven in Theorem \ref{thm-weak-mild} that being $u_n$ an $L^1$-mild solution, it is also a weak dual solution, hence we just have to take limits in the following equality:
\begin{equation*}
\int_{0}^{T} \int_{\hn} (- \Delta_{\hn})^{-s}(u_n) \, \frac{\partial \psi}{\partial t} \,
\dg \, {\rm d}t - \int_{0}^{T} \int_{\hn} u_n^m \psi \, \dg\, {\rm d}t = 0
\end{equation*}
for every test function $\psi \in C^1_c(0,T; \mbox{L}_c^{\infty}(\hn)).$ Having in mind this goal, we note that
$$\int_{0}^{T} \int_{\hn} (- \Delta_{\hn})^{-s}(u_n) \, \frac{\partial \psi}{\partial t} \,
\dg \, {\rm d}t =\int_{0}^{T} \int_{\hn} u_n \, (- \Delta_{\hn})^{-s}(\frac{\partial \psi}{\partial t}) \,
\dg \, {\rm d}t $$
and that, as $n\rightarrow + \infty$, we have
$$\int_{0}^{T} \int_{\hn} u_n \, (- \Delta_{\hn})^{-s}(\frac{\partial \psi}{\partial t}) \,
\dg \, {\rm d}t \rightarrow  \int_{0}^{T} \int_{\hn} u \, (- \Delta_{\hn})^{-s}(\frac{\partial \psi}{\partial t}) \,
\dg \, {\rm d}t\,,$$
since the sequence $\{u_n\}$ is monotone, $u \in L^1\left( (0, T) ; L^{1} (\hn) \right) $ and, by Lemma \ref{Green-operator-estimate}, $(- \Delta_{\hn})^{-s}(\frac{\partial \psi}{\partial t})$  $\in L^{\infty}((0,T)\times \hn)$.
Furthermore, as $n\rightarrow + \infty$, we have
$$\int_{0}^{T} \int_{\hn} u_n^m \psi \, \dg\, {\rm d}t \rightarrow \int_{0}^{T} \int_{\hn} u^m \psi \, \dg\, {\rm d}t$$
since, the sequence $\{u_n^m\}$ is monotone and $u^m \in L^1\left( (0, T) ; L^{1} (\hn) \right) $.
\end{itemize}
The proof of Theorem \ref{thm-existence} when $u_0\in L^1(\hn)$ is concluded.

\noindent\textbf{Proof of Theorem \ref{thm-existence}. The cases $u_0\in L^1_{\Phi}(\hn)$ with $\Phi\in \mathcal{W}$ and $u_0\in L^1_{\Phi_1}(\hn)$ . }The proof is similar to the one for $u_0\in L^1(\hn)$, hence we just sketch it. 

Again, the first step consist in showing that the WDS defined in \eqref{def.WDS.approx} is bounded, more precisely, by semicontinuity of the $L^\infty$-norm, we have that for all $t>0$
\begin{equation*}
\|u(t)\|_{L^\infty(\hn)}\leq \liminf_{n\rightarrow +\infty}\|u_n(t)\|_{L^\infty(\hn)}
\leq \ka_1\liminf_{n\rightarrow +\infty}\frac{\|u_{0,n}\|^{2s \vartheta_1}_{L^1_{\Phi_1}(\hn)}}{t^{N \vartheta_1}}
= \ka_1 \frac{\|u_{0}\|^{2s \vartheta_1}_{L^1_{\Phi_1}(\hn)}}{t^{N \vartheta_1}}\,,
\end{equation*}
and we have used the $L^1_{\Phi_1}-L^\infty$ smoothing estimates \eqref{thm.smoothing.HN-like.estimate}, which hold true for the mild solutions of the approximating sequence $u_n$, as shown in Theorem \ref{thm-weak-mild}. Similarly for $\Phi\in \mathcal{W}$, using the smoothing estimate \eqref{eq-k3}, we get
\begin{equation}\label{sm-phi}
\|u(t)\|_{L^\infty(\hn)}\leq \liminf_{n\rightarrow +\infty}\|u_n(t)\|_{L^\infty(\hn)}
\leq  \frac{\ka_5}{t^{1/m}}  \liminf_{n\rightarrow +\infty} \|u_{0,n}\|_{L^1_{\Phi}(\hn)}^{1/m}
=  \frac{\ka_5}{t^{1/m}}  \|u_0\|_{  L^1_{\Phi}(\hn)}^{1/m}\,.
\end{equation}
As a consequence, $u_{n}(t) \rightarrow u(t)$ strongly in $L^p_{\Phi_1}$ (resp. $L^p_{\Phi}$), for all $p\in (1,\infty)$ and $t>0$, as a consequence of the $L^1_{\Phi_1}$-strong convergence (resp. $L^1_{\Phi}$-strong convergence), using the weighted analogous of inequality \eqref{L1Lp} which has essentially the same proof.

We are finally ready to verify that $u$ is a weak dual solution in the sense of Definition~\ref{defi_WDS}:

\noindent\textbf{The case of $\Phi\in \mathcal{W}$. }

\begin{itemize}[leftmargin=*]
\item $u \in C([0, T) : L^{1}_{\Phi} (\hn))$.
This follows by noting that for all $0\leq \tau< t$ or $0\leq t<\tau$,  $x_0 \in \hn$ and $n\geq n_0>0$, one has
\[
\begin{split}
 \int_{\hn} | u(t,x)-u(\tau,x)|\, (\Phi \circ \tau_{-x_0})(x)\,& \dg(x) 
 \leq \int_{\hn} |u(t,x)-u_{n}(t,x)|  \, (\Phi \circ \tau_{-x_0})(x)\, \dg(x) \\
&   +2\int_{\hn} | u_{0,n}(x)-u_{0,n_0}(x)|  \, (\Phi \circ \tau_{-x_0})(x)\, \dg(x) \\
&
+\int_{\hn} |u_{n_0}(t,x)-u_{n_0}(\tau,x)|  \, (\Phi \circ \tau_{-x_0})(x)\, \dg(x) \\
& +\int_{\hn} |u_{n}(\tau,x)-u(\tau,x)|  \, (\Phi \circ \tau_{-x_0})(x)\, \dg(x)\,,
\end{split}
\]
 where we have also exploited the stability inequality given in Proposition \ref{monotonePhi}. Using that $u_{n}(t) \rightarrow u(t) $ strongly in $L^1_{\Phi \circ \tau_{-x_0}}$ and $u_{n_0} \in C([0, \infty) : L^1(\hn)) \subset C([0, T) : L^{1}_{\Phi} (\hn))$, by taking $n_0$ sufficiently large, the above estimate yields the continuity for $t>0$; when $t=0$ we proceed analogously.

\item $u^m \in L^1\left( (0, T) : L^{1}_{\Phi}(\hn) \right) $. This follows by the smoothing estimate \eqref{sm-phi}. Indeed, we have
\[\begin{split}
\int_0^T \int_{\hn} u(t,x)^m \,\Phi(x) \, \dgh(x)\,dt
&\leq \int_0^T\|u(t, x)\|_{L^\infty(\hn)}^{m-1}  \int_{\hn} u(t,x) \,\Phi(x) \, \dgh(x)\,dt\\
&\leq \|u_0(x)\|^{\frac{2m-1}{m}}_{L^1_{\Phi}(\hn)} \int_0^T    \frac{\ka_5^{m-1}}{t_0^{\frac{m-1}{m}}}  \,dt <+\infty\,,
\end{split}
\]
since $\frac{m-1}{m}<1$.
\item $u$ satisfies the following identity

\begin{equation*}
\int_{0}^{T} \int_{\hn} (- \Delta_{\hn})^{-s}(u) \, \frac{\partial \psi}{\partial t} \,
\dgh(x) \, {\rm d}t - \int_{0}^{T} \int_{\hn} u^m \psi \, \dgh(x)\, {\rm d}t = 0
\end{equation*}
for every test function $\psi \in C^1_c(0,T; \mbox{L}_c^{\infty}(\hn)).$ This follows by recalling that $u_n$ is a weak dual solution and then passing to the limit in the equality:
\begin{equation*}
\int_{0}^{T} \int_{\hn} (- \Delta_{\hn})^{-s}(u_n) \, \frac{\partial \psi}{\partial t} \,
\dgh(x) \, {\rm d}t - \int_{0}^{T} \int_{\hn} u_n^m \psi \, \dgh(x)\, {\rm d}t = 0
\end{equation*}
for every test function $\psi \in C^1_c(0,T; \mbox{L}_c^{\infty}(\hn)).$ To this aim, we note that
$$\int_{0}^{T} \int_{\hn} (- \Delta_{\hn})^{-s}(u_n) \, \frac{\partial \psi}{\partial t} \,
\dgh(x) \, {\rm d}t =\int_{0}^{T} \int_{\hn} u_n \, (- \Delta_{\hn})^{-s}(\frac{\partial \psi}{\partial t}) \,
\dgh(x) \, {\rm d}t $$
and that, as $n\rightarrow + \infty$, we have
$$\int_{0}^{T} \int_{\hn} u_n \, (- \Delta_{\hn})^{-s}(\frac{\partial \psi}{\partial t}) \,
\dgh(x) \, {\rm d}t \rightarrow  \int_{0}^{T} \int_{\hn} u \, (- \Delta_{\hn})^{-s}(\frac{\partial \psi}{\partial t}) \,
\dgh(x) \, {\rm d}t\,,$$
since the sequence $\{u_n\}$ is monotone, $u \in L^1\left( (0, T) ; L^{1}_{\Phi} (\hn) \right) $ and, arguing as in the proof of by Lemma \ref{Green-operator-estimate}, it readily follows that $|(- \Delta_{\hn})^{-s}(\frac{\partial \psi}{\partial t})|\leq C(t)\Phi(x)$ for some $C\in C^0_c(0,T)$.

Furthermore, as $n\rightarrow + \infty$, we have
$$\int_{0}^{T} \int_{\hn} u_n^m \psi \, \dgh(x)\, {\rm d}t \rightarrow \int_{0}^{T} \int_{\hn} u^m \psi \, \dgh(x)\, {\rm d}t$$
since the sequence $\{u_n^m\}$ is monotone and $u^m \in L^1\left( (0, T) ; L^{1} _{\Phi}(\hn) \right) $.
\end{itemize}

\noindent\textbf{The case of $\Phi_1$. }This case is completely analogous to the previous ones, hence we omit the proof.

\noindent\textbf{Uniqueness of Minimal Weak Dual Solutions. }

The proof is an adaptation to the present setting of the Proof of Theorem 4.5 of \cite{BV1}, we sketch it here for convenience of the reader. Here $\Phi$ will denote either an element of the class $\mathcal{L}$ or the ground state $\Phi_1$. Essentially, this uniqueness for the minimal WDS follows by the T-contraction property. Let $u$ be the WDS corresponding to the initial datum $u_0$, obtained as limit of the monotone non-decreasing sequence $u_k$. Consider another monotone non-decreasing sequence $0\le v_{0,k}\le v_{0,k+1}\le u_0$\,, with $v_{0,k}\in L^1(\hn)\cap L^\infty(\hn)$\,, monotonically converging from below to $u_0\in L^1_{\Phi}(\hn)$ in the strong $L^1_{\Phi}(\hn)$ topology. Repeating the construction, we obtain another weak dual solution $v(t,x)\in C^0([0,\infty)\,:\,L^1_{\Phi}(\hn))$. We want to show that $u=v$ by showing that $v\le u$ and then that $u\le v$. To prove that $v\le u$ we use the estimates
\begin{equation*}
\int_{\hn}\big[v_k(t,x)-u_n(t,x)\big]_+\dgh(x)\le \int_{\hn}\big[v_k(0,x)-u_n(0,x)\big]_+\dgh(x)
\end{equation*}
which hold for all $u_n(t,\cdot)$ and  $v_k(t,\cdot)$, see Theorem \ref{BCPP-Thm}. Letting $n\to \infty$ we get that
\begin{align*}
\lim_{n\to\infty}\int_{\hn}\big[v_k(t,x)-u_n(t,x)\big]_+\dgh(x)
&\le \lim_{n\to\infty}\int_{\hn}\big[v_k(0,x)-u_n(0,x)\big]_+\dgh(x) \\
&=\int_{\hn}\big[v_k(0,x)-u_0(x)\big]_+\dgh(x)=0
\end{align*}
\normalcolor since $v_k(0,x)\le u_0$ by construction. Therefore also $v_k(t,x)\le u(t,x)$ for $t>0$, so that in the limit $k\to \infty$ we obtain $v(t,x)\le u(t,x)$\,. The inequality $u\le v$ can be obtained simply by switching the roles of $u_n$ and $v_k$\,. \qed


\appendix

\section{Green function estimates in $\hn$}\label{s-Green-function}

This section is devoted to provide proper estimates of Green function $\GH(x, y)$ for the fractional laplacian on the hyperbolic space for $0 < s < 1.$ It is well-known that  the Green function
is given by the following explicit formula

\begin{equation*}
\GH (x, y) = \int_0^{+\infty} \frac{k_{\hn}(t,x,y)}{t^{1-s}}\,dt,
\end{equation*}
where $k_{\hn}(t,x,y)$ is the heat kernel for $- \Delta_{\hn}.$ Furthermore, $k_{\hn}$ satisfies the following estimates (see \cite{DA}): for $N \geq 2,$ there exist
some positive constants $A_N$ and $B_N$ such that
\begin{equation}\label{heat-estimates}
A_N h_N(t, x, y) \leq k_{\hn}(t, x, y) \leq B_N h_N(t, x, y) \quad \mbox{for all} \ t > 0 \ \mbox{and} \ x, y \in \hn,
\end{equation}
where $h_N (t, x, y)$ is given by
\begin{equation}\label{hN}
h_N(t, x, y) := h_N(t, r) = (4 \pi t)^{-\frac{N}{2}} e^{- \frac{(N-1)^2 t}{4} - \frac{(N-1) r}{2} - \frac{r^2}{4t}} ( 1 + r + t)^{\frac{(N-3)}{2}}(1 + r),
\end{equation}
where $r := r(x, y) = {\rm dist}(x, y)$.

\medskip

Using \eqref{heat-estimates} we shall obtain the following Green functions estimates:

\begin{lem}\label{thm-green}
Let $N \geq 3$ and  $0  < s < 1$. Then there exist $R_1, R_2, \underline C_1, \overline C_1, \underline C_2, \overline C_2>0$, only depending on $N$ and $s$, such that
\begin{equation}\label{G1}
\frac{\underline C_1}{(r(x, y))^{N -2s}} \leq \GH(x, y) \leq \frac{\overline C_1}{(r(x, y))^{N -2s}} \quad \text{ for all }(x, y) \in \hn: 0< r(x, y) \leq R_1;
\end{equation}
 and
\begin{equation}\label{G2}
  \frac{\underline C_2\,e^{-(N-1) r(x, y)}}{(r(x, y))^{1-s}} \leq \GH(x, y) \leq \frac{\overline C_2\, e^{-(N-1) r(x, y)}}{(r(x, y))^{1-s}} \quad \text{ for all }(x, y) \in \hn:r(x, y) \geq R_2.
\end{equation}

\end{lem}

\begin{proof}
Let $h_{N}(t,r)$ be as defined in \eqref{hN}, we compute

\begin{align}\label{C1}
& \int_0^{+\infty} \frac{h_{N}(t,r)}{t^{1-s}}\,dt = \int_{0}^{\infty} \frac{(4 \pi t)^{N/2}
e^{-\frac{(N-1)^2}{4}t -
 \frac{(N-1)r}{2}- \frac{r^2}{4t}}(1 + r + t)^{(N-3)/2}(1 + r)}{t^{1 -s}}\,{\rm d}t \notag \\
 & =  \frac{e^{-(N-1)r/2 }(1 + r)}{(4 \pi)^{N/2}}
 \int_{0}^{\infty} t^{-N/2 - 1 + s} (1 + r + t)^{(N-3)/2}
 e^{-(\frac{(N-1) \sqrt{t}}{2}  - \frac{r}{2 \sqrt{t}})^2}
 e^{-(N-1)r/2 } \, {\rm d}t \notag \\
 & = \frac{e^{-(N-1) r}(1 + r)}{(4 \pi)^{N/2}}
 \int_{0}^{\infty} t^{\frac{-N + 2s -2}{2}} (1 + r + t)^{(N-3)/2}
 e^{-(\frac{(N-1) \sqrt{t}}{2}  - \frac{r}{2 \sqrt{t}})^2} \, {\rm d}t.
\end{align}
Now making the following substitution one can write

$$
z = \frac{r}{2\sqrt{t}} - \frac{(N-1)\sqrt{t}}{2} \ \Leftrightarrow \sqrt{t}
= \frac{- z + \sqrt{z^2 + (N-1) r}}{N-1}.
$$

Therefore, we have

$$
{\rm d}t = -\frac{2}{(N-1)^2} \frac{( \sqrt{z^2 + (N-1)r} - z)^2}{\sqrt{z^2 + (N-1) r}} \, {\rm d}z.
$$
Further substituting back in \eqref{C1} yields

\begin{align}
& \int_0^{+\infty} \frac{h_{N}(t,r)}{t^{1-s}}\,dt = \frac{ e^{-(N-1)r}(1 + r)}{(N-1)^{2s - 3}(4\pi)^{N/2}} \times  \notag \\
& \left\{ \int_{-\infty}^{\infty} \frac{(1 + r + ( \sqrt{z^2 + (N-1)r} -z)^2 )^{(N-3)/2}  (\sqrt{z^2 + (N-1)r} - z)^2 e^{-z^2}}{(-z + \sqrt{z^2 + (N-1)r})^{N - 2s +2} \sqrt{z^2 + (N-1) r}} \, {\rm d}z \right\}. \notag
\end{align}
Now for $r \rightarrow \infty,$ one can look for the leading term as follows

\begin{align}
& \int_0^{+\infty} \frac{h_{N}(t,r)}{t^{1-s}}\,dt = C_N  e^{-(N-1)r}(1 + r)  r^{\frac{-N + 2s -2}{2} + \frac{N-3}{2} + \frac{1}{2}} \times  \notag \\
& \underbrace{ \left\{ \int_{-\infty}^{\infty} \frac{((1/r + 1)(N-1)^2 + ( \sqrt{z^2/r + (N-1)} -z/\sqrt{r})^2 )^{(N-3)/2}  (\sqrt{z^2/r + (N-1)} - z/\sqrt{r})^2
e^{-z^2}}{(-z/r + \sqrt{z^2 + (N-1)})^{N - 2s +2} \sqrt{z^2/r + (N-1)}} \, {\rm d}z \right\}}_\text{$= I_{N}$} . \notag
\end{align}
The finiteness of $I_N$ follows easily from the fact that $e^{-z^2}$ is the leading term at infinity.
Therefore, recalling  \eqref{heat-estimates}, we obtain
$$
\int_0^{+\infty} \frac{h_{N}(t,r)}{t^{1-s}}\,dt = C_N I_N e^{-(N-1)r} r^{s-1} + \circ(e^{-(N-1)r} r^{s-1}) \quad \text{as } r \rightarrow \infty.
$$
Hence, we obtain \eqref{G2}. Now we turn to prove \eqref{G1}. To see this we can make the change of
variable as  $\alpha = \frac{r^2}{4t}$ and therefore we can write

$$
\int_0^{+\infty} \frac{h_{N}(t,r)}{t^{1-s}}\,dt = C_N \frac{e^{-(N-1)r/2} (1 + r)}{r^{(N-2s)}}
\underbrace{ \int_{0}^{\infty} {\alpha}^{\frac{N-2s -2}{2}} e^{\frac{-(N-1)^2 r^2}{16 \alpha} - \alpha}
\left(1 + r + \frac{r^2}{4 \alpha}\right)^{\frac{(N-3)}{2}} \, {\rm d}\alpha}_\text{$= J_N$}.
$$
Again, as before, the integral is finite for $0 < s < 1$ and hence

$$
G_{s}(r, t) = C_N J_N r^{-(N - 2s)} + o(r^{-(N - 2s)})\quad \text{as } r \rightarrow 0^+.
$$
This proves the lemma.

\end{proof}
From Lemma \ref{thm-green} we derive
\begin{cor}\label{cor-1}
Let $N \geq 3$ and $ 0 < s < 1$. Then there exist positive constants $\underline C_3,\overline C_3, \underline C_4, \overline C_4$, only depending on $N$ and $s$, such that
\begin{equation}\label{G3}
\underline C_3 \frac{1}{(r(x, y))^{N-2s}}  \leq \GH(x, y) \leq \overline C_3 \frac{1}{(r(x, y))^{N-2s}} \quad \text{ for all }(x, y) \in \hn: 0<r(x, y) \leq 1\,,
\end{equation}
\begin{equation}\label{G4}
 \underline C_4 \frac{ e^{-(N-1) r(x, y)}}{(r(x, y))^{1-s}}  \leq  \GH(x, y) \leq \overline C_4 \frac{ e^{-(N-1) r(x, y)}}{(r(x, y))^{1-s}} \quad \text{ for all }(x, y) \in \hn: r(x, y) \geq 1\,.
\end{equation}
Furthermore, the following global estimate holds:
\begin{equation}\label{G5}
\GH(x, y) \leq \frac{\overline{C}_5}{(r(x, y))^{N-2s}} \quad \text{ for all }(x, y) \in \hn: r(x, y) > 0\,,
\end{equation}
where $\overline{C}_5=\min\{\overline C_3,\overline C_4\}$.
\end{cor}

\begin{proof}
The proof of this corollary is a straightforward consequence of Lemma~\ref{thm-green}. We give a detailed proof of \eqref{G4}, the proof of \eqref{G3} can be achieved with similar arguments. On the other hand, the proof of \eqref{G5} follows by noting that $e^{-(N-1)r}\leq r^{-N+s+1}$ for all $r>1$ and then combining \eqref{G3} with \eqref{G4}.
\par
\medskip
\textbf{Proof of \eqref{G4}.} We may assume $R_2 > 1,$ where $R_2$ is as given in Lemma~\ref{thm-green} (otherwise, there is nothing to prove).  Next we set
\begin{equation}\label{g}
g(r):= \frac{e^{-(N-1)r}}{r^{1-s}} \quad \text{with }r>0 \,
\end{equation}
and we note that
$$
g(r(x, y)) \geq g(R_2) \quad \text{for all }1 \leq r(x, y) \leq R_2\,.
$$
Set $M : = { \rm max}_{1 \leq r(x, y) \leq R_2} \GH(x, y)$, then
$$
\GH(x, y) \leq M = \frac{M}{g(R_2)} g(R_2) \leq \frac{M}{g(R_2)} g(r(x, y)) \quad \text{for all } 1 \leq r(x, y) \leq R_2.
$$
By taking $\overline C_4: = {\rm max} \{ \frac{M}{g(R_2)}, \overline C_2 \}$ we get the upper bound in \eqref{G4}.

Now, to obtain the lower bound, we note that
$$
g(r(x, y)) \leq g(1) \quad \text{for all }1 \leq r(x, y) \leq R_2\,.
$$
Then, we set $m : = { \rm min}_{1 \leq r(x, y) \leq R_2} \GH(x, y)$ and we conclude
$$
\GH(x, y) \geq m =\frac{m}{g(1)} g(1) \geq \frac{m}{g(1)} g(r(x, y)) \quad \text{for all } 1 \leq r(x, y) \leq R_2.
$$
Taking $\underline C_4 := {\rm min} \{ \frac{m}{g(1)}, \underline C_2 \},$ we obtain the lower bound in \eqref{G4}.

\end{proof}

It can be useful to check that, if $\psi \in C^1_c(0,T; \mbox{L}_c^{\infty}(\hn))$ with $\psi\geq 0$, then $(- \Delta_{\hn})^{-s}(\psi)$ behaves like $G_{s}(x_0,x)$ near infinity:
\begin{lem}\label{Green-operator-estimate}
For all $\psi \in C^1_c(0,T; \mbox{L}_c^{\infty}(\hn))$ with $\psi\geq 0$   and $x_0\in \hn$, there exist
$\hat R=\hat R(\psi,x_0) > 0$ and two positive functions $\underline C=\underline C(\psi,x_0),\overline C=\overline C(\psi,x_0) \in C^1_c(0,T)$ such that
\begin{equation}\label{est1}
\underline C(t)\,\frac{ e^{-(N-1) r(x_0, x)}}{(r(x_0, x))^{1-s}} \leq  (-\Delta_{\hn})^{-s}\,\psi(x,t) \leq \overline C(t)\, \frac{ e^{-(N-1) r(x_0, x)}}{(r(x_0, x))^{1-s}}  \quad  \text{ for } r(x_0,x)\geq  \hat R\,.
\end{equation}
Furthermore, for all $R>0$, there exist  two positive functions $\underline D=\underline D(R,\psi,x_0),\overline D=\overline D(R,\psi,x_0) \in C^1_c(0,T)$  such that
\begin{equation}\label{est2}
 \underline D(t) \leq  (-\Delta_{\hn})^{-s}\,\psi(x,t) \leq \overline D(t)  \quad  \text{ for } r(x_0,x)\leq R\,.
\end{equation}
\end{lem}

\begin{proof}
\textbf{Proof of \eqref{est1}.}
We first prove the upper bound. Let $g=g(r)$ be as in \eqref{g} and denote with $K\subset \hn$ the compact support of $\psi$ with respect to the space variable. Then we can write
\begin{align}\label{represent-1}
(-\Delta_{\hn})^{-s}\,\psi(x,t) &= \int_{K} \GH(x, y) \psi(y,t) \, \dg(y)
\\\notag &= g(r(x_0, x)) \int_{K} \frac{ \GH(x, y)}{g(r(x_0, x))} \, \psi(y,t) \, \dg(y)\,.
\end{align}
Furthermore, by
Corollary~\ref{cor-1} there exist constants $\underline C_4$ and $\overline C_4$
such that there holds
$$
\underline C_4\, g(r(x, y)) \leq \GH(x, y) \leq \overline C_4 \, g(r(x, y)) \quad \mbox{for} \  r(x, y) > 1.
$$
Let $\gamma > 0$ be such that $r(x_0, y) \leq \gamma$ for all $y \in K$ and assume $r(x_0, x) \geq 1 + \gamma,$  we have
$$
r(x,y)  \geq  r(x_0, x) - r(x_0, y) > 1 + \gamma - \gamma = 1.
$$
Hence,
$$
\frac{\GH(x, y)}{g(r(x_0, x))} \leq \overline C_4 \frac{g(r(x,y))}{g(r(x_0, x))}\quad \mbox{for} \ r(x_0, x) \geq 1 + \gamma.
$$
Now, using the fact that the function $g$ is decreasing, for $r(x_0, x) \geq  1 + \gamma$ and for all $y\in K$ we have
 \begin{align*}
 \frac{\GH(x, y)}{g(r(x_0, x))}&  \leq \overline C_4 \,  \frac{g(r(x, y))}{g(r(x_0, x))}
 \leq \overline C_4 \,  \frac{g(r(x_0, x) - r(x_0, y))}{g(r(x_0, x))} \\
 & = \overline C_4 \, e^{(N-1)r(x_0, y)} \left( \frac{r(x_0, x)}{r(x_0, x) - r(x_0, y)} \right)^{1-s} \\
 & = \overline C_4 \,  e^{(N-1)r(x_0, y)} \frac{1}{\left( 1- \frac{r(x_0, y)}{r(x_0, x)} \right)^{1-s}}
 \leq \overline C_4 \frac{e^{(N-1)r(x_0, y)}}{\left( 1- \frac{r(x_0, y)}{1 + \gamma} \right)^{1-s}}\,,
 \end{align*}
 where the latter term is well defined since $r(x_0, y) \leq \gamma$ for all $y \in K.$
Now, plugging the above estimate into \eqref{represent-1} we conclude that
$$
(-\Delta_{\hn})^{-s}\,\psi(x,t) \leq g(r(x_0, x)) \underbrace{\int_{K} \overline C_4 \,
\frac{e^{(N-1)r(x_0, y)}}{\left( 1- \frac{r(x_0, y)}{1 + \gamma} \right)^{1-s}} \,  \psi(y,t)\, \dg(y)}_{= \overline C(t)\in C^1_c(0,T)} \quad \text{for all } r(x_0, x) \geq 1+\gamma\,.
$$
\medskip

Next we turn to prove lower the bound in \eqref{est1}. For $r(x_0, x) \geq 1+\gamma$, invoking again Corollary~\ref{cor-1}, we have
\begin{align*}
\int_{K} \frac{\GH(x, y)}{g(r(x_0, x))} \, \psi(y,t) \, \dg(y) & \geq \underline C_4 \int_{K}
\frac{g(r(x, y))}{g(r(x_0, x))} \, \psi(y,t) \, \dg(y) \notag \\
& = \underline C_4 \int_{K} e^{(N-1) (r(x_0, x) - r(x, y))} \left( \frac{r(x_0, x)}{r(x, y)} \right)^{1-s} \, \psi(y,t) \,  \dg(y) \notag \\
& \geq \underline C_4 \int_{K} e^{-(N-1) r(x_0, y)} \left( \frac{r(x_0, x)}{r(x, y)} \right)^{1-s} \, \psi(y,t) \, \dg(y)\,,
\end{align*}
where the last estimate follows from
 $$
 r(x_0, x) - r(x, y) \geq - r(x_0, y)\,.
 $$
Assume now that $r(x_0, x) > 3\gamma$, recalling that $r(x_0, y) \leq \gamma,$ we infer
$$r(x, y) \geq r(x_0, x) - r(x_0, y)\geq 2\gamma$$
which yield
$$
\left( 1 -  \frac{r(x_0, y)}{r(x, y)} \right)^{1-s} \geq \frac{1}{2^{1-s}}.
$$
Inserting this into \eqref{represent-1} we finally obtain
$$
(-\Delta_{\hn})^{-s}\,\psi(x,t) \geq g(r(x_0, x)) \underbrace{\int_{K} \overline C_4 \,
\frac{e^{(N-1)r(x_0, y)}}{2^{1-s}} \,  \psi(y,t)\, \dg(y)}_{= \underline C(t)\in C^1_c(0,T)} \quad \text{for all } r(x_0, x) > 3\gamma.
$$
Summarising, the proof of \eqref{est1} follows by taking $\hat R:=\max\{1+\gamma, 3\gamma\} $.

\medskip

\textbf{Proof of \eqref{est2}.}
As above, let $K$ be the compact support of $\psi$ with respect to $x$ and let $\gamma > 0$ be such that $r(x_0, y) \leq \gamma$ for all $y \in K$. Let $R>0$, for all $r(x_0,x)\leq R$ we have $r(x,y)\leq r(x,x_0)+r(x_0,y)\leq R +\gamma$. Hence, by exploiting the estimate \eqref{G5}, we infer
\begin{align*}
(-\Delta_{\hn})^{-s}\,\psi(x,t) & = \int_{K} \GH(x, y) \psi(y,t) \dg(y) \\
& \leq \|\psi (x,t)\|_{L^{\infty}(\hn)} \int_{B_{R +\gamma}(x)}  \frac{\overline{C}_5}{(r(x, y))^{N-2s}} \, \dg(y)
\end{align*}
and the  upper bound in \eqref{est2}  proof follows by setting $$D(t):= \|\psi (x,t)\|_{L^{\infty}(\hn)} \omega_N \int_{0}^{R +\gamma}  \frac{\overline{C}_5}{r^{N-2s}} \,(\sinh(r))^{N-1}\, {\rm d}r\,$$ where $\omega_N$ is the volume of the $N$ dimensional unit sphere.\par
 
The lower bound in \eqref{est2}, readily follows from Corollary \ref{cor-1}, indeed we have
\begin{align*}
(-\Delta_{\hn})^{-s}\,\psi(x,t)& \geq  \underline C_3 \int_{K \cap B_1(x)}  \frac{1}{(r(x, y))^{N-2s}} \psi(y,t) \dg(y)
\\& +  \underline C_4 \int_{K \cap B_1^{c}(x)} \frac{ e^{-(N-1) r(x, y)}}{(r(x, y))^{1-s}} \psi(y,t) \dg(y) 
\\& \geq  \underline C_3 \int_{K \cap B_1(x)}  \psi(y,t) \dg(y)+  \underline C_4  \frac{ e^{-(N-1) (R+\gamma)}}{(R+\gamma)^{1-s}} \int_{K \cap B_1^{c}(x)} \psi(y,t) \dg(y) 
\\& \geq \min\{ \underline C_3, \underline C_4  \frac{ e^{-(N-1) (R+\gamma)}}{(R+\gamma)^{1-s}} \} \int_{K} \psi(y,t) \dg(y) :=\underline D(t)\,.
\end{align*}

\end{proof}


\par\bigskip\noindent
\textbf{Acknowledgments. }The second author is partially supported by Project MTM2017-85757-P (Spain), and by the E.U. H2020 MSCA programme, grant agreement 777822. The first and fourth author are partially supported by the PRIN project 201758MTR2 ``Direct and inverse problems for partial differential equations: theoretical aspects and applications'' (Italy) and they  are members of the Gruppo Nazionale per l'Analisi Matematica, la Probabilit\`a e le loro Applicazioni (GNAMPA) of the Istituto Nazionale di Alta Matematica (INdAM). The research of the third author is supported in part by an INSPIRE faculty fellowship (IFA17-MA98).


\end{document}